\documentclass{article}
\usepackage[utf8]{inputenc}
\usepackage{enumerate}
\usepackage{amssymb, bm}
\usepackage{amsthm}
\usepackage{amsmath}
\usepackage{tikz-cd} 
\usepackage{bbold}
 \usepackage[all]{xy}
\title{A study of nefness in higher codimension}
\author{Xiaojun WU}
\date{\today}
\newtheorem{mythm}{Theorem}
\newtheorem{mylem}{Lemma}
\newtheorem{myprop}{Proposition}
\newtheorem{mycor}{Corollary}
\newtheorem{mydef}{Definition}
\newtheorem{myrem}{Remark}
\begin{document}
\def\cI{\mathcal{I}}
\def\Z{\mathbb{Z}}
\def\Q{\mathbb{Q}}  \def\C{\mathbb{C}}
 \def\R{\mathbb{R}}
 \def\N{\mathbb{N}}
 \def\H{\mathbb{H}}
  \def\P{\mathbb{P}}
 \def\rC{\mathcal{C}}
  \def\nd{\mathrm{nd}}
  \def\d{\partial}
 \def\dbar{{\overline{\partial}}}
\def\dzbar{{\overline{dz}}}
 \def\ii{\mathrm{i}}
  \def\d{\partial}
 \def\dbar{{\overline{\partial}}}
\def\dzbar{{\overline{dz}}}
\def \ddbar {\partial \overline{\partial}}
\def\cN{\mathcal{N}}
\def\cE{\mathcal{E}}  \def\cO{\mathcal{O}}
\def\cF{\mathcal{F}}
\def\P{\mathbb{P}}
\def\cI{\mathcal{I}}
\def \loc{\mathrm{loc}}
\def \log{\mathrm{log}}
\def \cC{\mathcal{C}}
\bibliographystyle{plain}
\def \dim{\mathrm{dim}}
\def \RHS{\mathrm{RHS}}
\def \liminf{\mathrm{liminf}}
\def \ker{\mathrm{Ker}}
\def \Pic{\mathrm{Pic}}
\def \Alb{\mathrm{Alb}}
\def \tors{\mathrm{Tors}}
\def \ch{\mathrm{ch}}
\def \Id{\mathrm{Id}}
\def \td{\mathrm{td}}
\def \id{\mathrm{id}}
\maketitle
\begin{abstract}
In this work, following the fundamental work of Boucksom, we construct the nef cone of a compact complex manifold in higher codimension and give explicit examples for which these cones are different.
In the third and fourth sections, we give different versions of Kawamata-Viehweg vanishing theorems regarding nefness in higher codimension and numerical dimensions.
We also show by examples the optimality of the divisorial Zariski decomposition given in \cite{Bou04}. %which simplifies and generalises the example of Nakayama in \cite{Nak}. 
\end{abstract}
\section{Introduction}
\paragraph{}
One of the reformulations of the Kodaira embedding theorem is that
a compact complex manifold is projective if and only if the K\"ahler cone, i.e.\ the convex cone spanned by K\"ahler forms in $H^2(X,\mathbb{R})$, contains a rational point (i.e., an element in $H^2(X, \mathbb{Q})$).

As a general matter of fact, it is obviously interesting to study positive cones attached to compact complex manifolds and relate them with the geometry of the manifold.
In classical algebraic or complex geometry, the emphasis is on two types of positive cones: the nef and psef cones, defined as the closed convex cones spanned by nef classes and psef classes, respectively.
The nef cone is of course contained in the psef cone.

The work of Boucksom \cite{Bou04} defines and studies the so-called modified nef cone for an arbitrary compact complex manifold.
Thanks to this definition, Boucksom was able to show the existence of a divisorial Zariski decomposition for any psef class (i.e., any cohomology class containing a positive current). The modified cone just sits between the nef and psef cones.

Inspired by Boucksom's definition, in Section 2, we introduce the concept of a nef cone in arbitrary codimension for any compact complex manifold, which is an interpolation between the above positive cones.
\paragraph{}
\textbf{Definition A.}
Let $\alpha \in H^{1,1}_{BC}(X, \R)$ be a psef class. We say that $\alpha$ is nef in codimension $k$, if for any irreducible analytic subset $Z \subset X$ of codimension at most equal to $k$, we have the generic minimal multiplicity of $\alpha$ along $Z$ as (defined in \cite{Bou04})
$$\nu(\alpha, Z)=0.$$
\paragraph{}
With this terminology, the nef cone is the nef cone in codimension $n$, where $n$ is the complex dimension of the manifold, while the psef cone is the nef cone in codimension 0, and the modified nef cone is the nef cone in codimension 1.
We notice that the algebraic analogue in the projective case is introduced in \cite{Nak}.
In Section 4, we show that these cones are in general different and construct explicit examples where these cones are different.

Inspired by the work of \cite{Cao17} and using Guan-Zhou’s solution of Demailly’s strong openness conjecture, we get the following Kawamata-Viehweg vanishing theorem in Section 3. The proof follows Cao's proof closely:
\paragraph{}
\textbf{Theorem A.}
Let $(L,h)$ be  a  pseudoeffective  line  bundle  on  a  compact K\"ahler $n$-dimensional manifold $X$ with singular positive metric $h$.  Then the morphism induced by the inclusion $K_X \otimes L \otimes \cI(h) \to K_X \otimes L$
$$H^q(X,K_X \otimes L \otimes \cI(h)) \to H^q(X,K_X \otimes L )$$ vanishes for every $q \geq n-\nd(L) + 1$.
\paragraph{}
As an application, we obtain in Section 4 the following generalisation from the nef case to the psef case of a similar result stated in \cite{DP02}.
\paragraph{}
\textbf{Theorem B.}
Let $(X,\omega)$ be a compact K\"ahler manifold of dimension $n$ and $L$ a line bundle on $X$ that is nef in codimension 1. Assume that $\langle L^2 \rangle \neq 0$ where $\langle \bullet \rangle$ is the positive product defined in \cite{Bou02}. Assume that there exists an effective integral divisor $D$ such that $c_1(L)=c_1(D)$. Then
$$H^q(X,K_X+L) = 0$$
for $q \ge n-1$.
\paragraph{}
The proof of the above theorem is an induction on the dimension, using Theorem A.
A difference compared with the nef case treated in \cite{DP02} is that we need passing from an intersection number to a positive product (or movable intersection number), which is a non-linear operation. Nevertheless, under a condition of nefness in higher codimension, we get the following estimate.
\paragraph{}
\textbf{Lemma A.}
Let $\alpha$ be a nef class in codimension $p$ on a compact K\"ahler manifold $(X, \omega)$, then for any $k \le p$ and $\Theta$ any positive closed $(n-k,n-k)-$form we have
$$(\alpha^k, \Theta) \geq \langle \alpha^k, \Theta \rangle.$$
\paragraph{}
With this inequality, the intersection number calculation in \cite{DP02} is still valid, and thus the cohomology calculations can be recycled.

Observe that a current with minimal singularities need not have analytic singularities for every big class $\alpha$ that is nef in codimension one but not nef in codimension two; such an example was given by \cite{Nak}, and also observed by Matsumura \cite{Mat}. 

As a consequence of Matsumura's observation, the assumption of our Kawamata-Viehweg vanishing theorem that the line bundle is numerically equivalent to an effective integral divisor is actually required.
In the nef case considered in \cite{DP02}, the authors deduce from their assumption that the line bundle $L$ is nef with $(L^2) \neq 0$ that
$L$ is numerically equivalent to an effective integral divisor $D$, and that there exists a positive singular metric $h$ on $L$  such that $\cI(h)=\cO(-D)$.

However, for a big line bundle $L$ that is nef in codimension one but not nef in codimension two over an arbitrary compact K\"ahler manifold $(X, \omega)$, we have $\langle L^2 \rangle \neq 0$ and $\frac{i}{2 \pi} \Theta(L,h_{\min})$ need not be a current associated with an effective integral divisor. 

Another by-product is a (probably already known) example of a projective manifold $X$ with $-K_X$ psef, for which the Albanese morphism is not surjective.
It was proven in \cite{Cao13}, \cite{Paun12} (and \cite{Zha} for the projective case) that the Albanese morphism of a compact K\"ahler manifold with $-K_X$ nef is always surjective.
Thus replacing nefness by pseudoeffectivity in the study of Albanese morphism seems to be a non-trivial problem. 

\textbf{Acknowledgement} I thank Jean-Pierre Demailly, my PhD supervisor, for his guidance, patience and generosity. 
I would like to thank Junyan Cao, Henri Guenancia and Andreas Höring for some very useful suggestions on the previous draft of this work.
I would also like to express my gratitude to colleagues of Institut Fourier for all the interesting discussions we had. This work is supported by the PhD program AMX of \'Ecole Polytechnique and Ministère de l'Enseignement Supérieur et de la Recherche et de l’Innovation, and the European Research Council grant ALKAGE number 670846 managed by J.-P. Demailly.
We thank the anonymous reviewer for a very careful reading of this paper, and for insightful comments and suggestions.
\section{Nefness in higher codimension}
We first recall some technical preliminaries introduced in \cite{Bou04}. Throughout this paper, $X$ is assumed to be a compact complex manifold equipped with
some reference Hermitian metric $\omega$ (i.e.\ a smooth positive
definite $(1,1)$-form); we usually take $\omega$ to be
K\"ahler if $X$ possesses such metrics. The Bott-Chern cohomology group $H^{1,1}_{BC}(X, \R)$ is the space of $d$-closed smooth (1,1)-forms modulo
$i\d\dbar$-exact ones.
By the quasi-isomorphism induced by the inclusion of smooth forms
into currents, $H^{1,1}_{BC}(X, \R)$ can also be seen as the space of
$d$-closed $(1,1)$-currents modulo $i \d \dbar$-exact ones.
A cohomology class $\alpha \in H^{1,1}_{BC}(X,\R)$ is said to be pseudo-effective iff it contains a positive current; $\alpha$ is nef iff, for each $\varepsilon >0$, $\alpha$ contains a smooth form $\alpha_{\varepsilon}$ such that $\alpha_{\varepsilon} \geq -\varepsilon \omega$; $\alpha$ is big iff it contains a K\"ahler current, i.e.\ a closed $(1,1)$-current $T$ such that $T \geq \varepsilon \omega$ for $\varepsilon >0$ small enough. 

\begin{mydef}{\rm(\cite{DPS01})}
Let $\varphi_1,\varphi_2$ be two quasi-psh functions on $X$ $($i.e.\ $i \d \dbar \varphi_i \geq -C \omega$ in the sense of currents for some $C \geq 0)$. The function $\varphi_1$ is said to be less singular than $\varphi_2$ $($one then writes $\varphi_1 \preceq \varphi_2)$ if $\varphi_2 \leq \varphi_1+C_1$ for some constant $C_1$. Let $\alpha$ be a fixed psef class in $H^{1,1}_{BC}(X,\R)$. Given $T_1,T_2,\theta \in \alpha$ with $\theta$ smooth, and $T_i=\theta+i\d \dbar \varphi_i$ with $\varphi_i$ quasi-psh $(i=1,2)$, we write $T_1 \preceq T_2$ iff $\varphi_1 \preceq \varphi_2$ $($notice that for any choice of $\theta$, the potentials $\varphi_i$ are defined up to smooth bounded functions, since $X$ is compact$)$. If~$\gamma$ is a smooth real $(1,1)$-form on $X$, the collection of all potentials $\varphi$ such that $\theta +i \d \dbar \varphi \geq \gamma$ admits a minimal element $T_{min,\gamma}$ for the pre-order relation~$\preceq$, constructed as the semi-continuous upper envelope of the subfamily of potentials $\varphi\le 0$ in the collection. 
\end{mydef}
\begin{mydef} {\rm (Minimal multiplicities)}.
The minimal multiplicity at $x \in X$ of the pseudo-effective class $\alpha \in H^{1,1}_{BC}(X,\R)$ is defined as
$$\nu(\alpha, x) := \sup_{\varepsilon>0} \nu(T_{\min,\varepsilon},x)$$
where $T_{\min, \varepsilon}$ is the minimal element $T_{\min, -\varepsilon \omega}$ in the above definition and $\nu(T_{\min, \varepsilon},x)$ is the Lelong number of $T_{\min,\varepsilon}$ at $x$. When $Z$ is an irreducible analytic subset, we define the generic minimal multiplicity of $\alpha$ along $Z$as
$$\nu(\alpha, Z) := \inf\{\nu(\alpha, x),x \in Z\}.$$
\end{mydef}
When $Z$ is positive dimensional,
there exists for each $\ell\in \N^*$ a countable union of proper analytic subsets of $Z$ denoted by $Z_\ell= \bigcup_{p} Z_{\ell,p}$ such that $\nu(T_{\min, \frac{1}{\ell}}, Z):= \inf_{x \in Z} \nu(T_{\min, \frac{1}{\ell}},x)=\nu(T_{\min, \frac{1}{\ell}},x)$ for $x \in Z \smallsetminus Z_\ell$.
By construction, when $\varepsilon_1 < \varepsilon_2$, $T_{\min, \varepsilon_1} \succeq T_{\min, \varepsilon_2}$.
Hence for a very general point $x \in Z \smallsetminus\bigcup_{\ell \in \N^*}Z_\ell$,
$$\nu (\alpha,Z)\leq\nu(\alpha,x)=\sup_{\ell} \nu(T_{\min, \frac{1}{\ell}},Z).$$
On the other hand, for any $y \in Z$,
$$\sup_{\ell} \nu(T_{\min, \frac{1}{\ell}},Z)\leq \sup_{\ell} \nu(T_{\min, \frac{1}{\ell}},y)=\nu(\alpha,y).$$
In conclusion, $\nu(\alpha,Z)=\nu (\alpha, x)$ for a very general point $x \in Z \smallsetminus\bigcup_{\ell \in \N^*}Z_\ell$
and $\nu(\alpha, Z)=\sup_{\varepsilon} \nu (T_{\min, \varepsilon}, Z)$.

Now we can define the concept of nefness in higher codimension implicitly used in \cite{Bou04}. It is the generalisation of the concept of ``modified nefness'' to the higher codimensional case.
\begin{mydef}
Let $\alpha \in H^{1,1}_{BC}(X, \R)$ be a psef class. We say that $\alpha$ is nef in codimension $k$, if for every irreducible analytic subset $Z \subset X$ of codimension at most equal to $k$, we have
$$\nu(\alpha, Z)=0.$$
\end{mydef}
We denote by $\cN_k$ the cone generated by nef classes in codimension $k$.
By Proposition 3.2 in \cite{Bou04}, a psef class $\alpha$ is nef iff for any $x \in X$, $\nu(\alpha,x)=0$.
By our definition, psef is  equivalent to nef in codimension $0$, and nef is equivalent to nef in codimension $n:= \dim_{\C} X$. In this way, we get a bunch of positive cones on $X$, satisfying the inclusion relations
$$\cN=\cN_n \subset \cdots \subset \cN_1 \subset \cN_0=\cE$$
where $\cN$ and $\cE$ are cones of nef and psef classes respectively.
By a proof similar to those of Propositions 3.5, 3.6 in \cite{Bou04}, we get
\begin{myprop}
$(1)$ For every $x\in X$ and every irreducible analytic subset $Z$, the map $\cE \to \R^+$ defined on the cone $\cE$ of psef classes by
$\alpha \mapsto \nu(\alpha, Z)$ is convex and homogeneous. It is continuous on the interior $\cE^{\circ}$, and lower semi-continuous on the whole of $\cE$.
\vskip3pt
\noindent
$(2)$ If $T_{\min}\in\alpha$ is a positive current with minimal singularities,
we have $\nu(\alpha, Z)\leq \nu(T_{\min},Z)$. 
\vskip3pt
\noindent
$(3)$ If $\alpha$ is moreover big, we have $\nu(\alpha, Z)=\nu(T_{\min},Z)$.
\end{myprop}

\noindent
The following lemma is a direct application of the proposition.
\begin{mylem}
Let $Y \subset X$ be a smooth submanifold of $X$ and $\pi: \tilde{X} \to X$ be the blow-up of $X$ along~$Y$. We denote by $E$ the exceptional divisor. If $\alpha \in H^{1,1}_{BC}(X, \R)$ is a big class, we have
$$\nu(\alpha, Y)=\nu(\pi^* \alpha,E).$$
For $Z$ any irreducible analytic set not included in $Y$, we denote by $\tilde{Z}$ the strict transform of $Z$. Then
$$\nu(\alpha, Z)=\nu(\pi^* \alpha,\tilde{Z}).$$
For $W$ any irreducible analytic set in $Y$, we have
$$\nu(\alpha, W)=\nu(\pi^* \alpha,\P(N_{Y/X}|_W)).$$
\end{mylem}
\begin{proof}
Since $\alpha$ is big, we know that by taking a suitable regularisation, there exists a K\"ahler current $T\in\alpha$ with analytic singularities. The pullback $\pi^*T$ of this current is a smooth K\"ahler current on some dense open set $U$ where $\pi$ is a biholomorphism.
Hence the volume of $\pi^* \alpha$ defined as $\int_{T \in \pi^* \alpha, T \geq 0,} T_{ac}^n$ (ac means the absolute part of the current) is larger than the mass of $\pi^* T$ on $U$ which is strictly positive. By \cite{Bou02} $\pi^* \alpha$ is thus big.

By the proposition, we have
$$\nu(\alpha, Y)=\inf_{T \in \alpha}\nu(T,Y), \quad
\nu(\pi^* \alpha,E)=\inf_{S \in \pi^* \alpha} \nu(S,E). $$
On the other hand, the push forward and pull back operators acting on positive $(1,1)$ currents induce bijections between positive currents in the class $\alpha$ and positive currents in the class $\pi^* \alpha$.
Let $\theta \in \alpha$ be a smooth form such that $T= \theta + i \d \dbar \varphi$.
We recall that for any irreducible analytic set $W$ with local generators $(g_1, \cdots, g_r)$ near a regular point $w \in W$, 
the generic Lelong number along $W$ is the largest $\gamma$ such that $\varphi \leq \gamma \log(\sum |g_i|^2) + O(1)$ near $w$.
Since $\pi^* (g_1, \cdots, g_r) \cdot \cO_{\tilde{X}}= \cI_E$, we have
 $\nu(T, Y)=\nu(\pi^* T, E)$.
In particular, this implies that 
$$\nu(\alpha, Y)=\nu(\pi^* \alpha,E).$$
For $W$ any irreducible analytic set in the centre $Y$, since the exceptional divisor is isomorphic to $\P(N_{Y/X})$, the preimage of $W$ under the blow-up is isomorphic to $\P(N_{Y/X}|_W)$. In suitable local coordinates
$(z_1, \cdots, z_n)$ on $X$ and $(w_1, \cdots, w_n)$ on $\tilde X$, the
blow-up map is given by
$$\pi(w_1, \cdots,w_n )=(w_1, w_1 w_2, \cdots, w_1 w_s; w_{s+1}, \cdots, w_n).$$
In these coordinates, the centre $Y$ is given by the zero variety $V(z_{s+1}, \cdots, z_n)$. Assume that in this chart, $W=V(z_{s+1}, \cdots, z_n;f_1, \cdots, f_r)$ where $f_i$ is a function of $z_1, \cdots, z_s$ (as we can assume without loss of generality).
Then $$\pi^*(\cI_W) \cdot \cO_{\tilde{X}}=(w_1, f_1(w_1, w_1 w_2, \cdots, w_1 w_s), \cdots, f_r(w_1, w_1 w_2, \cdots, w_1 w_s))=\cI_{\P(N_{Y/X}|_W)}.$$
In particular, this implies that
$$\nu(\alpha, W)=\nu(\pi^* \alpha,\P(N_{Y/X}|_W)).$$

For the second statement, we just observe that the generic Lelong number along $Z$ (resp. $\tilde{Z}$) is equal to the Lelong number at some very general point.
Since $Z$ is not contained in $Y$, we can assume without loss of generality that the very general point is not in $Y$ (resp.\ $E$).
Since the Lelong number is a coordinate invariant local property, for such very general point $x \in \tilde{Z}$ near which $\pi$ is a local biholomorphism and any $T \in \alpha, T\geq 0 $, $\nu (T, Z)=\nu(T, \pi(x))=\nu(\pi^* T, x)=\nu (\pi^* T, \tilde{Z})$.
Hence we have
$$\nu(\alpha, Z)=\nu(\pi^* \alpha,\tilde{Z}).$$
\end{proof}
As a corollary, we find
\begin{mycor}{\it
Let $\mu: \tilde{X} \to X$ be a composition of finitely many blow-ups with smooth centres in~$X$. 
If $\alpha \in H^{1,1}_{BC}(X, \R)$ is a big class on $X$ such that $\mu^* \alpha$ is nef in codimension $k$, then $\alpha$ is a nef class in codimension~$k$.}
\end{mycor}

\begin{proof}
Without loss of generality, we can reduce ourselves to the case where $\mu$ is a blow-up of smooth centre $Y$ in $X$.
By Lemma 1, the generic minimal multiplicity of $\alpha$ along any irreducible analytic set of $X$ of codimension at most equal to $k$ is equal to the generic minimal multiplicity of $\mu^* \alpha$ along with a certain irreducible analytic set of $\tilde{X}$ of codimension at most equal to $k$.
So by the definition of nefness in codimension $k$, the fact $\mu^* \alpha$ is nef in codimension $k$ implies that $\alpha$ is nef in codimension $k$.
\end{proof}
\begin{myrem}
{\em
Let $X$ be a compact complex manifold whose big cone is non empty.
Recall that by Proposition 2.3 of \cite{Bou04}, a class $\alpha$ is modified K\"ahler (i.e.\ $\alpha$ is in the interior of nef cone in codimension 1) iff there exists a modification $\mu: \tilde{X} \to X$ and a K\"ahler class $\tilde{\alpha}$ on $\tilde{X}$ such that $\alpha=\mu_* \tilde{\alpha}$. As a consequence,
for $\mu: \tilde{X} \to X$ a modification between compact K\"ahler manifolds and
$\tilde{\alpha} \in H^{1,1}_{BC}(\tilde{X}, \R)$ a big and nef class on $\tilde{X}$ in codimension $k$, it is false in general that
$\mu_* \tilde{\alpha}$ is a nef class in codimension $k$.
}
\end{myrem}
To give an equivalent definition of nefness in higher codimension, we will need the following definition.
\begin{mydef}{\rm (Non-nef locus)}

The  non-nef  locus  of  a  pseudo-effective  class $\alpha \in H^{1,1}_{BC}(X,\R)$ is defined by
$$E_{nn}(\alpha) :=\{x\in X, \nu(\alpha, x)>0 \}.$$
\end{mydef}
\begin{myprop}
A psef class $\alpha$ is nef in codimension $k$ iff for any $\varepsilon >0$, any $c >0$, the codimension of any irreducible component of $E_c(T_{\min, \varepsilon})$ is larger than $k+1$.
\end{myprop} 
\begin{proof}
By the definition of non-nef locus, we have
$$E_{nn}(\alpha)=\bigcup_{\varepsilon>0} \bigcup_{c >0} E_{c}(T_{\min, \varepsilon})=\bigcup_{m \in \N^*} \bigcup_{n \in \N^*} E_{\frac{1}{n}}(T_{\min, \frac{1}{m}}).$$
We know by Siu's theorem \cite{Siu74} that $E_{\frac{1}{n}}(T_{\min, \frac{1}{m}})$ is an analytic set.
Hence the non-nef locus is a countable union of irreducible analytic sets.
If for any $\varepsilon >0$, any $c >0$, the codimension of any irreducible component of $E_c(T_{\min, \varepsilon})$ is larger than $k+1$, then for any irreducible analytic set $Z$ of codimension $k$, $E_{nn}(\alpha) \cap Z$ is strictly contained in $Z$. Hence $\nu(\alpha,Z)=0$.

On the other direction, assume there exists an irreducible component $Z$ of $E_{\frac{1}{n}}(T_{\min, \frac{1}{m}})$ has codimension at most equal to $k$. On each point $x$ of this irreducible component, $\nu(\alpha,x) \geq \nu(T_{\min, \frac{1}{m}},x) \geq \frac{1}{n}$.
In particular, $\nu(\alpha,Z)\geq \frac{1}{n}$, which contradicts the fact that $\alpha$ is nef in codimension $k$.
\end{proof}
\begin{myrem}
{\em
If the manifold $X$ is projective, it is enough to test the minimal multiplicity along irreducible analytic subsets of codimension $k$ to prove that the class is nef in codimension $k$.
The argument is as follows: 

For any irreducible analytic set $Z$ of codimension strictly smaller than $k$, for any $z \in Z$, since
$X$ is projective; there exists some hypersurfaces $H_i$ such that $z \in H_i$ and the irreducible component
of $Z \cap \bigcap_i H_i$ containing $z$ has codimension k. In other words, $Z$ is covered by the irreducible analytic subsets of
codimension exactly k. By assumption, the generic minimal multiplicity along any of these irreducible
analytic subsets is 0. This implies that the generic minimal multiplicity along $Z$ at most equal to the
generic minimal multiplicity along any of these irreducible analytic sets is 0.
}
\end{myrem} 
\begin{myrem}
{\em
In the general setting of compact complex manifolds, it is crucial to test the generic minimal multiplicity along any analytic set of codimension at most equal to $k$, instead of any analytic set of codimension $k$, to obtain the inclusion of the various positive cones.
The problem is that there may exist too few analytic subsets in an arbitrary compact complex manifold.

A typical example can be taken as follows.
For example, let $X_1$ be a compact manifold such that the nef cone is strictly contained in the psef cone (for example, we can take the projectivisation of an unstable rank two vector bundle over a curve of genus larger than 2, whose cones are explicit calculated on page 70 \cite{Laz}). Let $X_2$ be a very general torus such that the only analytic sets in $X_2$ are either union of points or $X_2$.
%We denote the dimension of $X_1, X_2$ as $n_1,n_2$.
Let $\beta$ be a psef but not nef class on $X_1$.
Let $X:= X_1 \times X_2$ with natural projections $\pi_1, \pi_2$ and $\alpha := \pi_1^* \beta$.
Assume that $\dim (X_1) < \dim (X_2)$.
Fix $\omega_1, \omega_2$ two reference Hermitian metrics on $X_1, X_2$.

Now $\alpha$ is a psef but not nef class on $X$.
The only analytic subsets of codimension $\dim (X_1)$ is the fibre of $\pi_2$.
$\alpha$ has generic minimal multiplicity 0 along any fibre of $\pi_2$. The reason is as follows:
The minimal current in $\alpha$ larger than $-\varepsilon(\pi_1^* \omega_1 +\pi^* \omega_2)$ denoting $\min\{T \in \alpha, T \geq -\varepsilon(\pi_1^* \omega_1 +\pi^* \omega_2)\}$ is less singular than the pull back of the minimal current in $\beta$ larger than $-\varepsilon \omega_1$ denoting $\min\{S \in \beta, S \geq -\varepsilon \omega_1 \}$ and the restriction of these minimal currents on the fibre of $\pi_2$ is trivial.
In other words, the generic Lelong number of $\min\{T \in \alpha, T \geq -\varepsilon(\pi_1^* \omega_1 +\pi^* \omega_2)\}$ along the fibres is smaller than the generic Lelong number of the pull back of $\min\{S \in \beta, S \geq -\varepsilon \omega_1 \}$ which is 0.
Hence it is itself 0.

On the other hand, for any positive integers $m,n$, take $Z$ a positive dimensional irreducible component of $E_{\frac{1}{n}}(T_{\min, \frac{1}{m}})$ in the non-nef locus of $\beta$.
The existence of such an irreducible component will be shown in Lemma 2, which implies that $\alpha$ has to be nef in codimension at most equal to $n-2$.
Now $Z \times X_2$ is an irreducible analytic set of codimension strictly smaller than $\dim(X_1)$.
But the generic minimal multiplicity along $Z \times X_2$ is larger than $\frac{1}{n}$.
In particular this shows that $\alpha$ is not nef in codimension $\dim(X_1)-\dim(Z)$.
}
\end{myrem}
\begin{myrem}
{\em
Let us mention that our definition of nefness in codimension 1 is equivalent to the definition of modified nefness.
By definition, a psef class is modified nef iff its generic minimal multiplicity is 0 along any prime divisor.
To prove the equivalence, we need to show that for any psef class $\alpha$ on $X$ we automatically have 
$$\nu(\alpha,X)=0.$$
It is because that $\nu(\alpha,X) \leq \nu(T_{\min}, X)$ where the latter is 0.
We notice that by Siu's decomposition theorem \cite{Siu74}, the set $E_{c>0}(T_{\min})=\bigcup_{n \in \N^*} E_{\frac{1}{n}}(T_{\min})$ is countable union of proper analytic sets.

By this observation, we can also say that the ``nef in codimension 0'' cone is exactly the psef cone. 
}
\end{myrem}

In analogy to the case of surfaces for which the nef cone coincides with the modified nef cone, the nef
cone in codimension $n -1$ coincides with the nef cone.
\begin{mylem}
Let $\alpha$ be a psef class, then $\alpha$ is nef in codimension $n - 1$ iff $\alpha$ is nef.
\end{mylem}
\begin{proof}
If $\alpha$ is nef, by inclusion of different positive cones, it is nef in codimension $n - 1$. On the
other direction, we will need the following proposition 3.4 in \cite{Bou04} which is a reformulation of a result of P\u{a}un \cite{Paun}.

A pseudo-effective class $\alpha$ is nef iff $\alpha|_Y$ is pseudo-effective for every irreducible analytic subset $Y \subset E_{nn}(\alpha)$.

Given a class $\alpha$ that is nef in codimension $n-1$, Proposition 2 implies that for any $\varepsilon >0$ and any $c >0$ the analytic set $E_c(T_{\min, \varepsilon}) $ is a finite set.
Therefore, the non-nef locus, which is a countable union of finite sets, has at most countably many points.
In particular, this implies that the restriction of $\alpha$ on any $Y \subset E_{nn}(\alpha)$ is 0, hence psef.
By the above proposition, $\alpha$ is nef.
\end{proof}
\begin{myrem}
{\em
Recall that a line bundle $L$ over a projective manifold is nef iff its intersection number
with any curve satisfies $(L \cdot C) \geq  0$. By the important work of \cite{BDPP}, a class is psef iff its pairing with
any movable curve is positive. Here a curve $C$ is said to be movable if $C = C_{t_0} $ is a member of
an analytic family $(C_t )_{t \in S}$ such that $\bigcup_{ t\in S} C_t = X$ and, as such, $C$ is a reduced irreducible 1-cycle.
Notice also that nef is equivalent to nef in codimension $n - 1$ and psef is equivalent to nef in
codimension 0.

Then it is natural to conjecture that a class over a projective manifold is nef in codimension $k$
if and only if its pairing with any movable curve in codimension k is positive. Here a curve $C$ is
said to be movable in codimension k if $C = C_{t_0}$ is a member of an analytic family $(C_t )_{t \in S}$ such
that $\bigcup_{ t \in S} C_t$ is an analytic subset of X of codimension k and, as such, $C$ is a reduced irreducible
1-cycle.
}
\end{myrem}
\begin{myrem}
{\em
Inspired by the result of P\u{a}un, it seems to be natural to conjecture that a psef class $\{T\}$ with $T$ a positive current on $X$ is nef in codimension $k$ if and only if that for any irreducible component of codimension at most $k$ in $\bigcup_{c >0}E_c(T)$ $\{T\}|_Z$ is nef in codimension $k-\mathrm{codim} (Z,X)$.
When $k=n$, this is exactly the result of P\u{a}un. When $k=0$, it is trivial. 
The ``only if'' part is quite similar. The restriction of the potentials of $T_{\min, \varepsilon}$ on any irreducible analytic set of codimension at most $k$ decreases to a potential on the submanifold.
If we fix the maximum of the potentials on $X$ to be 0, they form a compact family.
The limit potential would be quasi-psh, and thus the restriction of the class on the analytic set is psef. The ``if'' part is of course true if the manifold is a K\"ahler surface by P\u{a}un's result. 

The ``if'' part is also true for the case $k=1$ if the manifold is hyperk\"ahler. By Lemma 4.9 \cite{Bou04} (see also \cite{Huy}) a psef class $\alpha$ on a hyperk\"ahler manifold is modified nef if and only if for any prime divisor $D$ one has $q(\alpha,D) \geq 0$. Here, we let $\sigma$ be a non-trivial symplectic holomorphic form on $X$, and define
$$q(\alpha, \beta) :=\int_X \alpha \wedge \beta \wedge (\sigma \wedge \overline{\sigma})^{\frac{n}{2}-1}$$
to be the Beauville–Bogomolov quadratic form for any $(1,1)$-classes $\alpha, \beta$. For a psef $(1,1)$-class $\alpha$ such that $\alpha|_D$ is psef for any prime divisor $D$, we have
$$q(\alpha, \{[D]\})=\int_X \alpha \wedge \{[D]\} \wedge (\sigma \wedge \overline{\sigma})^{\frac{n}{2}-1}=\int_D \alpha  \wedge (\sigma \wedge \overline{\sigma})^{\frac{n}{2}-1}\ge 0.$$
Thus $\alpha$ is nef in codimension 1.

%The only non-trivial case is when $k=1$ in which case modified nef is equivalent to nef on a surface. For any irreducible curve $C$ $\alpha|_C$ is psef implies that $(\alpha \cdot C) \ge 0$. On the other hand, we have trivially that for any smooth K\"ahler form on $X$ $(\alpha \cdot C) \ge 0$. By Corollary 0.3 of \cite{DP} $\alpha$ is nef which finishes the proof. 

A natural idea to attack this question in general consists of extending the current on this subvariety $Z$ to $X$. If this is possible, the current with minimal singularity would have a potential larger than that of the extended current. In particular, the current with minimal singularity would have generic Lelong number 0 along $Z$.

In this direction, Collins and Tosatti proved the following results in \cite{CT1} and \cite{CT2}, which we now recall.

\begin{mythm}{\rm (Theorem 3.2 in \cite{CT2})}. Let $X$ be a compact Fujiki manifold and $\alpha$ a closed smooth real (1,1)-form on $X$ with $\{\alpha\}$ nef and $\int_X \alpha^n>0$.  
Let $E=V\cup \bigcup^I_{i=1} Y_i$ be an analytic subvariety of $X$, with $V,Y_i$ its irreducible components, and $V$ a positive dimensional compact complex submanifold of $X$.  
Let $R=\alpha+i\d \dbar F$ be a K\"ahler current in the class $\{\alpha\}$ on $X$ with analytic singularities precisely along $E$ 
and let $T=\alpha|_V+i\d \dbar \varphi$ 
be a K\"ahler current in the class $\{\alpha|_V\}$ on $V$ with analytic singularities. 
Then there exists a K\"ahler current $\tilde{T}=\alpha+i\d \dbar \Phi$ in the class $\{\alpha\}$ on $X$ with $\tilde{T}|_V$ smooth in a neighbourhood of the very general point of $V$.
\end{mythm}

\begin{mythm}{\rm (Theorem 1.1 in \cite{CT1})}. Let $(X,\omega)$ be a compact K\"ahler manifold and let $V\subset X$ be a positive-dimensional compact complex submanifold. 
Let $T$ be a K\"ahler current with analytic singularities along $V$ in the K\"ahler class $\{\omega|_V\}$.  Then there exists a K\"ahler current $\tilde{T}$ on $X$ in the class $\{\omega\}$ with $T= \tilde{T}|_V$.
\end{mythm}

Using their results, in a given K\"ahler class, one can extend K\"ahler currents with analytic singularities defined in a smooth subvariety. If the class is just nef and big on the K\"ahler manifold, one can only show the existence of a K\"ahler current whose potential is not identically infinity along the submanifold. Following Example 5.4 in \cite{BEGZ}, one can show that in nef and big class on a K\"ahler manifold $X$, one cannot always extend a positive current along a submanifold into a positive current on $X$. In their example, the positive current on the submanifold can even be chosen to be smooth.
More precisely there exists $C$, a submanifold of a certain compact K\"ahler manifold $X$, $\{\alpha\}$ a nef and big class on $X$ with a smooth representative $\alpha$ and $\varphi \in L^1_{\loc}(C)$ with $\alpha|_C +i \d \dbar \varphi \geq 0$, such that there does not exist a $\psi \in L^1_{\loc}(X)$ satisfying $\alpha+i \d \dbar \psi \geq 0$ and $\psi|_C=\varphi$.

Let us start the construction of the example.
Let $C$ be an elliptic curve and let $A$ be an ample divisor on $C$.
Let $V$ be the rank two vector bundle over $C$  the unique non-trivial extension of $\cO_C$.
Define $X:=\P(V \oplus A)$ and $\{\alpha\}:=c_1(\cO_X(1))$ with smooth representative $\alpha$.
Then $\cO_X(1)$ is a big and nef line bundle over $X$.
The quotient map $V\oplus A \to \cO_C$ induces a closed immersion $C \to X$.
In particular, we have $\cO_X(1)|_C= \cO_C$.
Since $c_1(\cO_X(1)|_C)=0$, there exists a smooth function $\varphi$ on $C$ such that $\alpha|_C +i \d \dbar \varphi=0$.
We prove by contradiction that there does not exist $\psi \in L^1_{\loc}(X)$ such that $\alpha+i \d \dbar \psi \geq 0$ and $\psi|_C=\varphi$.
The quotient map $V \oplus A \to V$ induces a closed immersion $\P(V) \to X$.
On the contrary, we would have $\alpha|_{\P(V)}+i \d \dbar \psi|_{\P(V)} \geq 0$ in the class $c_1(\cO_{\P(V)}(1))$.
By the calculation made in Example 1.7 of \cite{DPS94}, we know that
$$\alpha|_{\P(V)}+i \d \dbar \psi|_{\P(V)}=[C]$$
where $[C]$ is the current associated with $C$.
In particular, this shows that $\psi|_C \equiv- \infty$, a contradiction.

In other words, Theorem 1.1 of \cite{CT1} cannot be strengthened to obtain an extension of an arbitrary closed positive current in a class that is merely nef and big. Similarly, one cannot drop the K\"ahler current condition in the theorem of \cite{CT2}.

Let us return to our previous question. In order to get an analogue of P\u{a}un's result, the above discussion shows that we need to generalise Theorem 3.2 of \cite{CT2} to the class of a big class that is nef in codimension $k$ by adding a small K\"ahler form to the class and by using the semi-continuity of the generic minimal multiplicity. Unfortunately, we do not know how to do it at this point.
}
\end{myrem}
\section{Junyan Cao's and Guan-Zhou's vanishing theorem}
Following the ideas of \cite{DP02}, we get a K\"ahler version of the Kawamata-Viehweg vanishing theorem in the next section.
To prepare for the proof, we give a version of Junyan Cao's and Guan-
Zhou's vanishing theorem in terms of numerical dimension of line bundle instead of the numerical dimension of singular metric.
In \cite{Cao17}, Junyan Cao has proven the following Kawamata-Viehweg-Nadel type vanishing theorem.
\begin{mythm}
Let $(L,h)$ be  a  pseudo-effective  line  bundle  on  a  compact K\"ahler $n$-dimensional manifold $X$.  Then 
$$H^q(X,K_X \otimes L \otimes \cI(h)) = 0$$ for every $q \geq n-\nd(L,h) + 1$.
\end{mythm}
The numerical dimension $\nd(L,h)$ used in Cao's theorem is the numerical dimension of the closed positive $(1,1)$-current $i \Theta_{L,h}$ defined in his paper. Since we will not need this definition, we refer to his paper for further information.
We just recall the remark on Page 22 of \cite{Cao17}.
In Example 1.7 of \cite{DPS94}, they consider the nef line bundle $\cO(1)$ over the projectivisation of a rank two vector bundle over the elliptic curve $C$ which is the only non-trivial extension of $\cO_C$. They prove that there exists a unique positive singular metric $h$ on $\cO(1)$.
For this metric, $\nd(\cO(1),h)=0$.
But the numerical dimension of $\cO(1)$ is equal to 1. 
We recall that for a nef line bundle $L$ the numerical dimension is defined as
$$\nd(L):= \max \{p ; c_1(L)^p \neq 0\}.$$
We also remark that Cao’s technique of proof actually yields the result for the upper semi-continuous regularization of multiplier ideal sheaf defined as
$$\cI_+(h) := \lim_{\varepsilon \to 0}\cI(h^{1+\varepsilon})$$
instead of $\cI(h)$, but we can apply Guan-Zhou’s Theorem  \cite{GZ14a}\cite{GZ13} \cite{GZ15} \cite{GZ14b} to see that the equality $\cI_+(h) =\cI(h)$ always holds. 
In particular, by the Noetherian property of ideal sheaves, we have
$$\cI_+(h)=\cI(h^{\lambda_0} )=\cI(h)$$
for some $\lambda_0 >1$.
This fact will also be used in our result.

Here we prove the following version of Junyan Cao's and Guan-Zhou's vanishing theorem, following closely the ideas of Junyan Cao \cite{Cao17} and the version that was a bit simplified in \cite{Dem14}.
\begin{mythm}
Let $L$ be  a  pseudo-effective  line  bundle  on  a  compact K\"ahler $n$-dimensional manifold $X$.  Then the morphism induced by inclusion $K_X \otimes L \otimes \cI(h_{\min}) \to K_X \otimes L$
$$H^q(X,K_X \otimes L \otimes \cI(h_{\min})) \to H^q(X,K_X \otimes L )$$ is 0 map for every $q \geq n-\nd(L) + 1$.
\end{mythm}
\begin{myrem}
{\rm
In Example 1.7 of \cite{DPS94}, since the rank two vector bundle is the only non-trivial extension of $\cO_C$, there exists a surjective morphism from this vector bundle to $\cO_C$ which induces a closed immersion $C$ into the ruled surface.
The only positive metric on $\cO(1)$ has curvature $[C]$ the current associated to $C$.
On the other hand, $\cO(1)=\cO(C)$. 
%and $K_X=\cO(-2C)$.
So we have
$H^2(X, K_X \otimes \cO(1))=H^0(X, \cO(-1))=H^0(X, \cO(-C))=0$
and $H^2(X, K_X \otimes \cO(1) \otimes \cI(h_{\min}))=H^2(X, K_X \otimes \cO(1) \otimes \cO(-C))=H^0(X, \cO_X)=\C$.
This shows that to get a numerical dimension version of theorem, the best that we can hope for is that the morphism is 0 map instead of that $H^q(X, K_X \otimes L \otimes \cI(h_{\min}))=0$.
We notice that, in general, one would expect the vanishing result
$$H^q(X, K_X \otimes L)=0$$
for $q \geq n - \nd(L)+1$, whenever $L$ is a nef line bundle. Here the difficulty is to prove a general K\"ahler version since the results follow easily from an inductive hyperplane section argument when $X$ is projective (cf., e.g. Corollary (6.26) of \cite{Dem12}).
}
\end{myrem}
The K\"ahler version of the definition of numerical dimension is stated in \cite{Dem14}. For $L$ a psef line bundle on a compact K\"ahler manifold $(X, \omega)$, we define
$$\nd(L): = \max\{p \in [0,n] ;\exists c >0,\forall \varepsilon >0, \exists h_{\varepsilon},i\Theta_{L,h_{\varepsilon}} \geq -\varepsilon \omega ,\mathrm{such} \;\mathrm{ that}
\int_{X\setminus Z_{\varepsilon}}( i \Theta_{L,h_{\varepsilon}}+\varepsilon \omega)^p \wedge \omega^{n-p} \geq c\}.$$
Here the metrics $h_{\varepsilon}$ are supposed to have analytic singularities and $Z_{\varepsilon}$ is the singular set of the metric.

By the following remark, we can even assume that $h_{\varepsilon}$ as stated in the definition of the numerical dimension is increasing to $h_{\min}$ as $\varepsilon \to 0$.
What we need here is that the weight functions $\varphi_{\varepsilon}$ has limit $\varphi_{\min}$ and is pointwise at least equal to $\varphi_{\min}$ with a universal upper bound on $X$.
\begin{myrem}{\rm
$ T_{\min, \varepsilon' \omega} \preceq T_{\min, \varepsilon \omega}+(\varepsilon'-\varepsilon)\omega $ for any $\varepsilon \leq \varepsilon'$.
Denote $T_{\min, \varepsilon \omega}= \theta+\varepsilon \omega +i \d \dbar \varphi_{\min, \varepsilon \omega}$.
We can arrange that 
$$\varphi_{\min,0} \leq \varphi_{\min,\varepsilon \omega} \leq \varphi_{\min,\varepsilon' \omega}.$$
The Bergman kernel regularisation perserves the ordering of potentials (cf. \cite{Dem14}), 
so we have
$$\varphi_{0, \delta} \leq \varphi_{\varepsilon, \delta} \leq \varphi_{\varepsilon',\delta}.$$ for any $\delta >0$.
If $\delta(\varepsilon)$ is increasing with respect to $\varepsilon$, %by the proof of the proposition, 
we can choose the metric $h_{\varepsilon}$ to be decreasing with respect to $\varepsilon$.
The limit of $\varphi_{\varepsilon, \delta(\varepsilon)}$ as $\varepsilon \to 0$ is equal to $\varphi_{\min,0}$ corresponding to the metric with minimal singularities on $L$.}
\end{myrem}
Before giving the proof of the vanishing theorem, we give the general lines of the ideas and compare them with Cao's and Guan-Zhou's theorem.
The idea is using the $L^2$ resolution of the multiplier ideal sheaf and proving that every $\dbar$-closed $L^2(h_{\min})$ global section can be approximated by $\dbar$-exact $L^2(h_{\infty})$ global sections with $h_\infty$ some smooth reference metric on $L$.
To prove it, we solve the $\dbar$-equation using a Bochner technique with the error term (as in \cite{DP02}), and we prove that the error term tends to 0.

For this proposal, we need to estimate the curvature asymptotically by some special approximating hermitian metrics constructed using the Calabi-Yau theorem.
Cao tried to prove that the error term tends to be 0 in the topology induced by the $L^2$-norm with respect to the given singular metric. In this way, he tried to keep the multiplier ideal sheaf unchanged when approximating the singular metric utilizing suitable ``equisingular approximation".
For our proposal, we try to prove that the error term tends to be 0 in the topology induced by $L^2$-norm with respect to some (hence any) smooth metric. 
It would be enough for us that the multiplier ideal sheaf of $h_{\min}$ is included in the multiplier ideal sheaf of the approximating hermitian metric.
In some sense, Cao's theorem is more precise in studying the singularity of the metric, which somehow explains why his approach works for any singular metric while our approach applies only to the image of the natural inclusion.  

We start the proof of the vanishing theorem by the following technical curvature and singularity estimate.
\begin{myprop}
Let $(L,h_{\min})$ be a pseudo-effective line bundle on a compact K\"ahler manifold $(X,\omega)$. Let us write $T_{\min}=\frac{i}{2\pi}\Theta_{L,h_{\min}}=\alpha+\frac{i}{2\pi}\d \dbar \varphi_{\min}$ where $\alpha$ is the curvature of some smooth metric $h_{\infty}$ on $L$ and $\varphi_{\min}$ is a quasi-psh potential. Let $p=\nd(L)$ be the numerical dimension of $L$. Then, for every $\gamma\in{}]0,1]$ and $\delta\in{}]0,1]$, there exists a quasi-psh potential $\Phi_{\gamma,\delta}$ on $X$ satisfying the following properties$\,:$
\item{\rm(a)} $\Phi_{\gamma,\delta}$ is smooth in the complement $X\setminus Z_{\delta}$ of an analytic set $Z_{\delta}\subset X$.
\vskip2pt
\item{\rm(b)} $\alpha+\delta\omega+\frac{i}{2 \pi} \d \dbar \Phi_{\gamma,\delta}\geq \frac{\delta}{2}(1-\gamma)\omega$ on $X$.
\vskip2pt
\item{\rm(c)} $(\alpha+\delta\omega+\frac{i}{2 \pi} \d \dbar\Phi_{\gamma,\delta})^n\geq a\,\gamma^n\delta^{n-p}\omega^n$ on $X\setminus Z_\delta$.
\vskip2pt
\item{\rm(d)} $\sup_X\Phi_{1,\delta}=0$, and for all $\gamma\in{}]0,1]$ there are estimates $\Phi_{\gamma,\delta}\le A$ and
$$\exp\big(-\Phi_{\gamma,\delta}\big)\le 
e^{-(1+b\delta)\varphi_{\min}}\exp\big(A-\gamma\Phi_{1,\delta}\big)$$
\item{\rm(e)} For $\gamma_0,\,\delta_0>0$ small,
$\gamma\in{}]0,\gamma_0]$, $\delta\in{}]0,\delta_0]$, we have
$$\cI_+(\varphi_{\min})=\cI(\varphi_{\min})\subset \cI(\Phi_{\gamma,\delta}).$$
Here $a,\,b,\,A,\,\gamma_0,\,\delta_0$ are suitable constants independent of $\gamma$, $\delta$.
\end{myprop}
\begin{proof}
Denote by $\psi_{\varepsilon}$ the (non-increasing) sequence of weight functions as stated in the definition of numerical dimension.
We have $\psi_{\varepsilon}\ge\varphi_{\min}$ for all $\varepsilon >0$, the $\psi_{\varepsilon}$ have analytic singularities and 
$$\alpha+\frac{i}{2 \pi}\d\dbar \psi_{\varepsilon}\ge-\varepsilon \omega.$$
Then for $\varepsilon\le\frac{\delta}{4}$, we have
\begin{align*}
   \alpha+\delta\omega+\frac{i}{2\pi}\d \dbar \big((1+b\delta)\psi_{\varepsilon}\big)  &\geq \alpha+\delta\omega-(1+b\delta)(\alpha+\varepsilon\omega) \\
   &\ge\delta\omega-(1+b\delta)\varepsilon\omega-b\delta\alpha\ge
{\textstyle\frac{\delta}{2}}\omega
\end{align*}
for $b \in ]0, \frac{1}{5}]$ small enough such that $\omega -b \alpha \ge 0$.

Let $\mu:\widehat X\to X$ be a log-resolution of $\psi_{\varepsilon}$, so that 
$$\mu^*\big(\alpha+\delta\omega+\frac{i}{2 \pi} \d \dbar ((1+b\delta)\psi_{\varepsilon})\big)=[D_{\varepsilon}]+\beta_{\varepsilon}$$
where $\beta_{\varepsilon}\ge\frac{\delta}{2}\mu^*\omega\ge 0$ is a smooth closed $(1,1)$-form on $\widehat X$ that is strictly positive in the complement $\widehat X\setminus E$ of the exceptional divisor, and $D_{\varepsilon}$ is an effective $\R$-divisor that includes all components $E_\ell$ of $E$. The map $\mu$ can be obtained by Hironaka \cite{Hir64} as a composition of a sequence of blow-ups with smooth centres, and we can even achieve that $D_{\varepsilon}$ and $E$ are normal crossing divisors.
For arbitrary small enough numbers $\eta_\ell>0$, $\beta_\varepsilon-\sum\eta_\ell[E_\ell]$ is a K\"ahler class on~$\widehat X$.
Hence we can find a quasi-psh potential $\widehat\theta_{\varepsilon}$ on $\widehat X$ such that $\widehat\beta_{\varepsilon}:=\beta_\varepsilon-\sum\eta_\ell[E_\ell]+\frac{i}{2 \pi} \d \dbar \widehat\theta_\varepsilon$ is a K\"ahler metric on~$\widehat X$.
By taking the $\eta_\ell$ small enough, we may assume that
$$\int_{\widehat X}(\widehat\beta_\varepsilon)^n\ge
\frac{1}{2}\int_{\widehat X}\beta_\varepsilon^n.$$
We will use Yau's theorem \cite{Yau78} to construct a form in the cohomology class of $\widehat{\beta_\varepsilon}$ with better volume estimate.
We have
\begin{align*}
\alpha+\delta\omega+\frac{i}{2 \pi}\d \dbar\big((1+b\delta)\psi_{\varepsilon}\big)
&\ge\alpha+\varepsilon \omega+\frac{i}{2 \pi}\d \dbar \psi_{\varepsilon}+(\delta-\varepsilon)\omega
-b\delta(\alpha+\varepsilon\omega)\\
&\ge(\alpha+\varepsilon \omega+\frac{i}{2 \pi} \d \dbar \psi_{\varepsilon})+\frac{\delta}{2}\omega.
\end{align*}
The assumption on the numerical dimension of $L$ implies the existence of a constant $c>0$ such that, with $Z=\mu(E)\subset X$, we have
\begin{align*}
\int_{\widehat X}\beta_\varepsilon^n&=
\int_{X\setminus Z}\big(\alpha+\delta\omega+\frac{i}{2 \pi} \d \dbar ((1+b\delta)\psi_{\varepsilon})\big)^n \\
&\ge{n\choose p}\Big(\frac{\delta}{2}\Big)^{n-p}\int_{X\setminus Z}
\big(\alpha+\varepsilon\omega+\frac{i}{2\pi} \d \dbar \psi_{\varepsilon}\big)^p\wedge\omega^{n-p}
\ge  c\delta^{n-p}\int_X\omega^n.
\end{align*}
Therefore, we may assume
$$\int_{\widehat X}(\widehat\beta_\varepsilon)^n\ge \frac{c}{2}\,\delta^{n-p}\int_X\omega^n.$$
We take $\widehat{f}$ a volume form on $\widehat{X}$ such that $\widehat f>\frac{c}{3}\delta^{n-p}\mu^*\omega^n$ everywhere on $\widehat X$ and such that $\int_{\widehat X} \widehat f=\int_{\widehat X}\widehat\beta_\varepsilon^n$.
By Yau's theorem \cite{Yau78}, there exists a quasi-psh potential $\widehat\tau_\varepsilon$ on $\widehat X$ such that 
\hbox{$\widehat\beta_\varepsilon+\frac{i}{2 \pi} \d \dbar \widehat\tau_\varepsilon$} is a K\"ahler metric on $\widehat X$ with the prescribed volume form $\widehat f>0$.

Now push our focus back to $X$. Set $\theta_\varepsilon=\mu_*\widehat\theta_\varepsilon$ and
$\tau_\varepsilon=\mu_*\widehat\tau_\varepsilon\in L^1_{\loc}(X)$.
We define
$$\Phi_{\gamma,\delta}:=(1+b\delta)\psi_{\varepsilon}+\gamma(\theta_\varepsilon+\tau_\varepsilon).$$
By construction it is smooth in the complement $X\setminus Z_{\delta}$ i.e. property (a).
It satisfies
$$\mu^*\big(\alpha+\delta\omega+\frac{i}{2 \pi} \d \dbar ((1+b\delta)\psi_{\varepsilon}+
\gamma(\theta_\varepsilon+\tau_\varepsilon))\big)
=[D_\varepsilon]+(1-\gamma)\beta_\varepsilon+\gamma\Big(
\sum_\ell\eta_\ell[E_\ell]+\widehat\beta_\varepsilon+\frac{i}{2\pi} \d \dbar \widehat\tau_\varepsilon\Big)$$
$$\ge(1-\gamma)\beta_\varepsilon \ge\frac{\delta}{2}(1-\gamma)\,\mu^*\omega$$
since \hbox{$\widehat\beta_\varepsilon+\frac{i}{2 \pi} \d \dbar \widehat\tau_\varepsilon$} is a K\"ahler metric on $\widehat X$.
Thus the property (b) is satisfied.
Putting $Z_\delta=\mu(|D_\varepsilon|)\supset \mu(E)=Z$, we have on $X \setminus Z_\delta$
$$
\mu^* \big(\alpha+\delta\omega+\frac{i}{2 \pi}\d \dbar \Phi_{\gamma,\delta}\big)^n \ge \big(\beta_\varepsilon+\gamma\frac{i}{2 \pi} \d \dbar(\widehat \theta_\varepsilon+ \widehat\tau_\varepsilon) \big)^n
$$$$\ge\gamma^n\,(\widehat\beta_\varepsilon+\frac{i}{2 \pi} \d \dbar \widehat\tau_\varepsilon)^n\ge \frac{c}{3}\,\gamma^n\delta^{n-p}\mu^*\omega^n.
$$
Since $\mu: \widehat{X} \setminus D_\varepsilon \to X \setminus Z_\delta$ is a biholomorphism, the condition (c) is satisfied if we set $a=\frac{c}{3}$.

We adjust constants in $\widehat\theta_\varepsilon+\widehat\tau_\varepsilon$ so that $\sup_X\Phi_{1,\delta}=0$.
Since $\varphi_{\min}\le\psi_{\varepsilon}\le\psi_{\varepsilon_0}\le A_0:=\sup_X\psi_{\varepsilon_0}$ for $\varepsilon \leq \varepsilon_0$,
$$
\Phi_{\gamma,\delta}=(1+b\delta)\psi_{\varepsilon}+\gamma
\big(\Phi_{1,\delta}-\psi_{\varepsilon}\big)\geq 
(1+b\delta)\varphi_{\min}+\gamma\Phi_{1,\delta}
-\gamma A_0$$
and we have $\Phi_{\gamma,\delta}\le(1-\gamma+b\delta)A_0$.
Thus the property (d) is satisfied if we set $A:= (1+b)A_0$.

We observe that $\Phi_{1,\delta}$ satisfies $\alpha+\omega+dd^c\Phi_{1,\delta}\ge 0$ and $\sup_X\Phi_{1,\delta}=0$, hence $\Phi_{1,\delta}$ belongs to a compact 
family  of quasi-psh functions. 
By Theorem 2.50 a uniform version of Skoda’s integrability theorem in \cite{GZ17}, there exists a uniform small constant $c_0>0$ such that $\int_X
\exp(-c_0\Phi_{1,\delta})dV_\omega<+\infty$ for all $\delta\in{}]0,1]$.
If $f\in\cO_{X,x}$ is a germ of holomorphic function and $U$ a small
neighbourhood of~$x$, the H\"older inequality combined with estimate~(d)
implies
$$
\int_U|f|^2\exp(-\Phi_{\gamma,\delta})dV_\omega\le
e^A\Big(\int_U|f|^2e^{-p(1+b\delta)\varphi_{\min}}dV_\omega\Big)^{\frac{1}{p}}
\Big(\int_U|f|^2e^{-q\gamma\Phi_{1,\delta}}dV_\omega\Big)^{\frac{1}{q}}.
$$
Take $p\in{}]1,\lambda_0[$ (say $p=(1+\lambda_0)/2$),  and take
$$\gamma\le\gamma_0:=\frac{c_0}{q}=c_0\frac{\lambda_0-1}{\lambda_0+1}\quad
\hbox{and}\quad\hbox{$\delta \le\delta_0\in{}]0,1]$ so small that 
$p(1+b\delta_0)\le\lambda_0$.}$$
Then  $f\in\cI_+(\varphi_{\min})=\cI(\lambda_0\varphi_{\min})$ implies $f\in\cI(\Phi_{\gamma,\delta})$ which proves the condition (e).
\end{proof}
The rest of the proof follows from the proof of \cite{Cao17} (cf. also \cite{Dem14}, \cite{DP02}, \cite{Mou}).
We will just give an outline of the proof for completeness.

Let $\{f\}$ be a cohomology class in the group $H^q(X,K_X\otimes L\otimes\cI(h_{\min}))$, $q\ge n-\nd(L)+1$. 
The sheaf $\cO(K_X\otimes L)\otimes\cI(h_{\min})$ can be resolved by the complex $(K^{\bullet},\dbar)$ where $K^i$ is the sheaf of $(n,i)$-forms $u$ such that both $u$ and $\dbar u$ are locally $L^2$ with respect to the weight $\varphi_{\min}$. 
So $\{f\}$ can be represented by a $(n,q)$-form $f$ such that both $f$ and $\dbar f$ are $L^2$ with respect to the weight $\varphi_{\min}$,
i.e.\ $\int_X|f|^2\exp(-\varphi_{\min})dV_\omega<+\infty$ and
$\int_X|\dbar f|^2\exp(-\varphi_{\min})dV_\omega<+\infty$. 

We can also equip $L$ by the hermitian metric $h_\delta$ defined by the quasi-psh weight $\Phi_\delta=\Phi_{\gamma_0,\delta}$ obtained in Proposition 3, with $\delta\in{}]0,\delta_0]$.
Since $\Phi_\delta$ is smooth on $X\setminus Z_\delta$, the Bochner-Kodaira inequality shows that for 
every smooth $(n,q)$-form $u$ with values in $K_X\otimes L$ that is 
compactly supported on $X\setminus Z_\delta$, we have
$$\Vert\dbar u\Vert_\delta^2+\Vert\dbar^* u\Vert_\delta^2\ge
2\pi\int_X(\lambda_{1,\delta}+\ldots+\lambda_{q,\delta}-q\delta)|u|^2e^{-\Phi_\delta}dV_\omega,$$
where $\Vert u\Vert_\delta^2:=\int_X|u|^2_{\omega,h_\delta}dV_\omega=
\int_X|u|_{\omega, h_{\infty}}^2e^{-\Phi_\delta}dV_\omega$.
Condition (b) of Proposition 3 shows that
$$
0<\frac{\delta}{2}(1-\gamma_0) \leq \lambda_{1,\delta}(x)\le\ldots\le\lambda_{n,\delta}(x)
$$
where $\lambda_{i, \delta}$ are at each point $x\in X$, the eigenvalues of $\alpha+\delta\omega+\frac{i}{2 \pi} \d \dbar \Phi_\delta$ with respect to the base K\"ahler metric~$\omega$.
In other words, we have up to a multiple $2 \pi$
$$
\Vert\dbar u\Vert_\delta^2+\Vert\dbar^* u\Vert_\delta^2+\delta\Vert u\Vert_\delta^2\ge
\int_X(\lambda_{1,\delta}+\ldots+\lambda_{q,\delta})|u|_{\omega, h_{\infty}}^2e^{-\Phi_\delta}dV_\omega.
$$
By the proof of Theorem 3.3 in \cite{DP02}, we have the following lemma:
\begin{mylem}
For every $L^2$ section of $\Lambda^{n,q}T^*_X
\otimes L$ such that $\Vert f\Vert_\delta<+\infty$ and $\dbar f=0$ in the sense of distributions, there exists a $L^2$ section $v=v_\delta$ of $\Lambda^{n,q-1}T^*_X\otimes L$ and a $L^2$ section $w=w_\delta$ of $\Lambda^{n,q}T^*_X\otimes L$ such that $f=\dbar v+w$ with
$$
\Vert v\Vert_\delta^2+\frac{1}{\delta}\Vert w\Vert_\delta^2\le
\int_X\frac{1}{\lambda_{1,\delta}+\ldots+\lambda_{q,\delta}}|f|^2e^{-\Phi_\delta}dV_\omega.
$$
\end{mylem}
By Lemma 3 and Condition (d) of Proposition 3, the error term $w$ satisfies the $L^2$ bound,
$$\int_X|w|_{\omega, h_{\infty}}^2e^{-A}dV_\omega\le
\int_X|w|_{\omega, h_{\infty}}^2e^{-\Phi_\delta}dV_\omega\le\int_X\frac{\delta}{\lambda_{1,\delta}+\ldots+\lambda_{q,\delta}}|f|_{\omega, h_{\infty}}^2e^{-\Phi_\delta}dV_\omega.
$$
We will show that the right-hand term tends to 0 as $\delta \to 0$.
To do it, we need to estimate the ratio function $\rho_{\delta}:=\frac{\delta}{\lambda_{1,\delta}+\ldots+\lambda_{q,\delta}}$.
The ratio function is first estimated in \cite{Mou}.

By Estimates (b,c) in Proposition 3, we have $\lambda_{j,\delta}(x)\ge \frac{\delta}{2}(1-\gamma_0)$ and
$\lambda_{1,\delta}(x)\ldots \lambda_{n,\delta}(x)\ge a\gamma_0^n\delta^{n-p}$.
Therefore we already find $\rho_\delta(x)\le 2/q(1-\gamma_0)$.
On the other hand, we have
$$
\int_{X\setminus Z_\delta}\lambda_{n,\delta}(x)dV_\omega\le
\int_X(\alpha+\delta\omega+dd^c\Phi_\delta)\wedge\omega^{n-1}
=\int_X(\alpha+\delta\omega)\wedge\omega^{n-1}
\le{\rm Const},
$$
therefore the ``bad set'' $S_\varepsilon\subset X\setminus Z_\delta$ of points $x$ where $\lambda_{n,\delta}(x)>\delta^{-\varepsilon}$ has a volume with respect to $\omega$ Vol$(S_\varepsilon)\le C\delta^\varepsilon$ converging to~$0$ as $\delta\to 0$.
Outside of $S_\varepsilon$, 
$$
\lambda_{q,\delta}(x)^q\delta^{-\varepsilon(n-q)}\ge \lambda_{q,\delta}(x)^q\lambda_{n,\delta}(x)^{n-q}\ge a\gamma_0^n\delta^{n-p}.
$$
Thus we have $\rho_\delta(x)\le C\delta^{1-\frac{n-p+(n-q)\varepsilon}{q}}$. 
If we take $q\ge n-\nd(L)+1$ and $\varepsilon>0$ small enough, the exponent of $\delta$ in the final estimate is strictly positive.
Thus there exists a subsequence $(\rho_{\delta_\ell})$, $\delta_\ell\to 0$, that tends almost everywhere to $0$ on~$X$.

Estimate (e) in Proposition 3 implies the
H\"older inequality
$$
\int_X\rho_\delta|f|_{\omega, h_{\infty}}^2\exp(-\Phi_{\delta})dV_\omega\le
e^A\Big(\int_X\rho_\delta^p|f|_{\omega, h_{\infty}}^2e^{-p(1+b\delta)\varphi_{\min}}dV_\omega\Big)^{\!\frac{1}{p}}
\Big(\int_X|f|_{\omega, h_{\infty}}^2e^{-q\gamma_0\Phi_{1,\delta}}dV_\omega\Big)^{\!\frac{1}{q}}
$$
for suitable $p,q >1$ as in the proposition.
$|f|^2_{\omega, h_{\infty}} \leq C$ for some constant $C >0$ since $X$ is compact.
Taking $\delta \to 0$ yields that $w_{\delta} \to 0$ in $L^2( h_{\infty})$ by Lebesgue dominating theorem.

$H^q(X,K_X \otimes L)$ is a finite-dimensional Hausdorff vector space whose topology is induced by the $L^2$ Hilbert space topology on the space of forms.
In particular, the subspace of coboundaries is closed in the space of cocycles.  
Hence $f$ is a coboundary that completes the proof.

For any singular positive metric $h$ on $L$, by definition, $h$ is more singular that $h_{\min}$ which implies that $\cI(h) \subset \cI(h_{\min})$.
A direct corollary of the above theorem is the following.
\begin{mycor}
Let $(L,h)$ be  a  pseudo-effective  line  bundle  on  a  compact K\"ahler $n$-dimensional manifold $X$.  Then the morphism induced by inclusion $K_X \otimes L \otimes \cI(h_{}) \to K_X \otimes L$
$$H^q(X,K_X \otimes L \otimes \cI(h_{})) \to H^q(X,K_X \otimes L )$$ is 0 map for every $q \geq n-\nd(L) + 1$.
\end{mycor}
\section{Kawamata-Viehweg vanishing theorem}
We first give a ``numerical dimension version'' of the Kawamata-Viehweg vanishing theorem in the projective case.
In the following, we study various properties of nef classes in higher codimension. Then we end the section with a numerical version of the Kawamata-Viehweg vanishing theorem in the K\"ahler case.

To start with, we need the relation between movable intersection defined in \cite{BDPP}, \cite{Bou02b} and intersection number.
\begin{mylem}
Let $\alpha$ be a nef class in codimension $p$ on a compact K\"ahler manifold $(X, \omega)$. Then for any $k \le p$ and $\Theta$ any positive closed $(n-k,n-k)$-form we have
$$(\alpha^k, \Theta) \geq \langle \alpha^k, \Theta \rangle.$$
\end{mylem}
Here we use the definition of movable intersection defined in \cite{Bou02b} and \cite{BDPP}.
The movable intersection number $ \langle \alpha^k, \Theta \rangle$ in \cite{Bou02b} is defined as the limit for $\varepsilon >0$ converging to $0$ of the quantity:
$$\sup_{T_i} \int_{X \smallsetminus F} (T_1+\varepsilon \omega) \wedge \cdots (T_k +\varepsilon \omega) \wedge \Theta$$
where $T_i$ ranges all closed current with analytic singularities in the class $\alpha$ such that $T_i \geq - \varepsilon \omega$ and $F$ is the union of all singular part of $T_i$.
(In \cite{Bou02b}, the movable intersection number is defined for any closed positive current $\Theta$. In the following, we will take $\Theta$ to be $\omega^{n-k}$. Thus we consider only the case when $\Theta$ is a  positive closed form.)

The proof of the boundedness of the quantity is a consequence of regularisation and the theory of Monge-Amp\`ere operator. 
In the general case, we approximate the current $T_i$ decreasingly by the smooth forms by \cite{Dem82} with a uniform lower bound $-C \omega$ depending on $(X, \omega)$ and $\{T_i\}$.
Now on $X \smallsetminus F$ the current $(T_1+C \omega) \wedge \cdots (T_k +C \omega) \wedge \Theta$ is the limit of corresponding terms changing $T_i$ by its smooth approximation,  using the continuity of Monge-Amp\`ere operator with respect to decreasing sequence. However, the integral on $X \smallsetminus F$ obtained for the smooth approximation is bounded by its integral on $X$, which is the intersection number of cohomology classes $\{T_i +C \omega\}$ and $\{\Theta\}$. 
\begin{proof}
Our observation is that with better regularity on the cohomology class $\alpha$, we can define directly the Monge-Amp\`ere operator on $X$. 
So comparing to the general case, we can skip the approximation process and get rid of the dependence of $C$, which only depends on $(X, \omega)$ and $\alpha$ but not explicitly.

We recall the following Theorem (4.6) on the Monge-Amp\`ere operators in Chapter 3 of \cite{agbook}. 

Let $u_1,\cdots,u_q$ be quasi-plurisubharmonic functions on $X$ and $T$ be a closed positive current of bidimension $(p,p)$.  The currents $u_1 i \d \dbar u_2 \wedge \cdots \wedge i \d \dbar u_q \wedge T$ and $i \d \dbar u_1 \wedge i \d \dbar u_2 \wedge \cdots \wedge i \d \dbar u_q \wedge T$ are well defined and have locally finite mass in $X$ as soon as $q \le p$ and 
$$H_{2p-2m+1}(L(u_{j_1})\cap \cdots \cap L(u_{j_m}) \cap \mathrm{Supp} (T))= 0$$
for all choices of indices $j_1< \cdots < j_m$ in $\{1,\cdots,q\}$.

Here $H_{2p-2m+1}$ means the $(2p-2m+1)$-dimensional Hausdorff content of the subset of $X$ seen as a metric space induced by the K\"ahler metric. The unbounded locus $L(u)$ is defined to be the set of points $x \in X$ such that $u$ is unbounded in every neighbourhood of $x$.
When $u$ has analytic singularities, it is the singular part of $u$ (i.e. $\{u=-\infty\}$).

%From the theorem (4.9), we know the Monge-Amp\`ere operator is continuous with respect to decreasing sequence of functions.

Now return to the proof of the lemma. 
By definition $T_{i, \min,-\varepsilon \omega}$ is less singular than $T_i$.
Since for any $c >0$, $E_c(T_{i, \min,-\varepsilon \omega})$ has codimension larger than $p+1$,
the singular set of $T_i$ which has analytic singularities is also of codimension larger than $p+1$.
By Theorem (4.6) cited above, the current $(T_1+\varepsilon \omega) \wedge \cdots (T_k +\varepsilon \omega) \wedge \Theta$ is well-defined on $X$.
Thus we have
\begin{align*}
\int_{X \smallsetminus F} (T_1+\varepsilon \omega) \wedge \cdots (T_k +\varepsilon \omega) \wedge \Theta &\le \int_{X } (T_1+\varepsilon \omega) \wedge \cdots (T_k +\varepsilon \omega) \wedge \Theta\\
& = (\alpha+\varepsilon \{\omega\}) \cdot \cdots (\alpha+\varepsilon \{\omega\}) \cdot \{\Theta\}.
\end{align*}
Taking $\varepsilon \to 0$, we get
$(\alpha^k, \Theta) \geq \langle \alpha^k, \Theta \rangle.$
\end{proof}
\begin{myrem}
{\em
From the definition of movable intersection, it is easy to see that the movable intersection is superadditive and that $\langle L^p \rangle \neq 0$ implies that $\nd(L) \geq p$.
In fact, the inverse direction also holds.
If $(X, \omega)$ is compact K\"ahler, there exists $C>0$ such that $C \{\omega\}-c_1(L)$ contains a closed positive current.
Thus we have that
$$\langle c_1(L)^p , C^{\nd(L)-p} \omega^{n-p} \rangle \geq \langle c_1(L)^{\nd(L)}, \omega^{n-\nd(L)} \rangle \neq 0.$$
In particular, $\nd(L)$ is the maximum $p$ such that $\langle c_1(L)^p \rangle \neq 0$.
}
\end{myrem}
In the projective case, we can now give the following version of the Kawamata-Viehweg theorem in terms of nefness in higher codimension. The simple proof given below has been suggested to us by Demailly.
\begin{mythm}
Let $X$ be a projective manifold and $L$ a nef line bundle in codimension $p-1$. If $\langle c_1(L)^p \rangle \neq 0$, then for any $q \ge n-p+1$ we have
$$H^q(X, K_X \otimes L)=0.$$
\end{mythm}
\begin{proof}
The proof is an induction on the dimension of $X$.
Let $A$ be an ample divisor on $X$ and $\omega \in c_1(A)$ be a K\"ahler form.
Let $Y \in |kA|$ be a generic smooth hypersurface.
With the choice of $k$ big enough, we can assume that $H^q(X, L^{-1} \otimes \cO(-Y))=0$ for any $q<n$ by Kodaira vanishing theorem.
By Serre duality, the statement of the theorem is equivalent to prove that for any $q \le p-1$, we have
$$H^q(X,  L^{-1})=0.$$

Consider the long exact sequence associated to the short exact sequence
$$0 \to L^{-1} \otimes \cO(-Y) \to L^{-1} \to L^{-1}|_Y \to 0.$$
It turns out that it is enough to prove that $H^q(Y, L^{-1})=0$ for any $q \le p-1$.

We check that conditions are preserved under the intersection with a generic hypersurface.
Since $\alpha$ is nef in codimension $p-1$, we find that any irreducible component of 
$$E_{nn}(\alpha)=\bigcup_{m \in \N^*} \bigcup_{n \in \N^*} E_{\frac{1}{n}}(T_{\min, \frac{1}{m}}).$$
has codimension larger than $p$.
By regularisation of $T_{\min, \frac{1}{m}}$, there exists currents $T_m$ with analytic singularities in $\alpha$ larger that $- \frac{2}{m} \omega$.
Any irreducible component of the singular set of these currents has codimension larger than $p$.
For generic $Y$, the restriction of these currents on $Y$ is well defined for any $m$.
Since the inclusion of analytic sets is a Zariski closed condition, for generic $Y$, we can also assume that the singular set of $T_m$ is not contained in $Y$ for any $m$.

On the other hand, in the class $\alpha|_Y$, the current with minimal singularities that admits a lower bound $-\frac{2}{m} \omega|_Y$ is certainly less singular than $T_m|_Y$.
The upper-level set of the Lelong number of these minimal currents is included in the singular set of $T_m|_Y$, so it has codimension larger than $p$.
This means that $\alpha|_Y$ is nef in codimension $p-1$.

The condition $\langle \alpha^p \rangle \neq 0$ implies that 
$$\int_X \langle \alpha^p \rangle \wedge \omega^{n-p} >0.$$
In other words, there exist a sequence of currents with analytic singularities $T_m \in \alpha$ such that $T_m \geq - \frac{1}{m} \omega$ and
$$ \int_{X \smallsetminus F_m} (T_m+\frac{1}{m} \omega)^p \wedge \omega^{n-p} >c$$
for some $c >0$ independent of $m$ where $F_m$ is the singular set of $T_m$.

With a generic choice of $Y$, we can still assume that the restriction of $T_m$ is current with analytic singularities.
They satisfy the conditions $T_m|_Y \geq - \frac{1}{m} \omega|_Y$ and
$$ \int_{Y \smallsetminus F_m} (T_m|_Y+\frac{1}{m} \omega|_Y)^p \wedge \omega^{n-p-1} >\frac{c}{k}.$$
In other words, $\langle \alpha|_Y^p \rangle \neq 0$.

By induction on the dimension, we are reduced to proving the case where $X$ has dimension $p$, and $L$ is nef in codimension $p-1$, in which case $L$ is (plainly) nef by Lemma 2.
The condition of the movable intersection reduces to $\langle c_1(L)^p \rangle \neq 0$.
By Lemma 4, this implies that $(L^p) >0$.
In particular, $L$ is a nef and big line bundle.
Now the vanishing of cohomology classes follows from the classical Kawamata-Viehweg theorem.
\end{proof}
As pointed out to us by A. H\"oring, this can also be proven using the result of \cite{Kur13}.
\begin{myrem}
{\em 
When $p=n$, the above theorem is the classical Kawamata-Viehweg vanishing theorem for nef and big line bundle.
We notice that $\langle c_1(L)^n \rangle =\mathrm{Vol}(L)$ by Theorem 3.5 of \cite{BDPP}.
When $p=1$, the theorem states that if $L$ is a psef line bundle with $\langle c_1(L) \rangle \neq 0$, then $H^n(X, K_X \otimes L)=0$.
This case is trivial by the following easy lemma.
The first interesting case is when $L$ is nef in codimension 1 and  $\langle c_1(L)^2 \rangle \neq 0$.
In the following example, we show that we can not weaken the condition to the case that $L$ is only psef and  $\langle c_1(L)^2 \rangle \neq 0$.
On the other hand, by the divisorial Zariski decomposition, we can write any psef line bundle numerically as a sum of a nef class in codimension one and an effective class.
This shows that, in some sense, this kind of theorem is the best we can hope for.

Now we begin our example. Let $V$ be the unique non-trivial rank 2 extension of $\cO_C$ over an elliptic curve $C$. Let $X$ be the blow-up of a point of $\P(V) \times \P^1$ and $L$ be the pull back of $\cO_{\P(V)}(1) \boxtimes \cO_{\P^1}(1)$.
$\cO_{\P(V)}(1)$ is a nef line bundle.
We also notice that $c_1(\cO_{\P(V)}(1))^2=0$ and $c_1(\cO_{\P(V)}(1)) \neq 0$.
So $L$ is a nef line bundle over $X$ and $\nd(L)=2$.
By the above theorem we have that $H^{2}(X, K_X+L)=0$.
Let $E$ be the exceptional divisor of the blow-up.
The short exact sequence
$$0 \to K_X+L \to K_X+L+E \to K_X+L+E|_E \to 0$$
induces the long exact sequence
$$H^{2}(X, K_X+L) \to H^{2}(X, K_X+L+E) \to H^{2}(E, K_X+L+E|_E)=H^{0}(E, -L) \to H^{3}(X, K_X+L).$$
By Serre duality and the following lemma, $H^{3}(X, K_X+L)=H^{0}(X, -L)=0$.
Since $L|_E=\cO_E$, $H^{0}(E, -L) \cong \C$.
Thus we have that 
$$H^{2}(X, K_X+L+E) \cong \C.$$
Now $L+E$ is a psef line bundle over $X$ and $\nd(L+E)\geq 2$ but $H^{2}(X, K_X+L+E) \neq 0.$
The reason of the numerical dimension is as follows.
By the super-additivity of movable intersection, we have that
$$\langle (L+E)^2 \rangle \geq \langle L^2 \rangle+\langle E^2 \rangle+2\langle L \cdot E \rangle \geq \langle L^2 \rangle.$$ 
}
\end{myrem}
\begin{mylem}
Let $(L,h)$ be a non-trivial $($i.e. $L\neq\cO_X)$ psef line bundle over a compact complex manifold~$X$.
Then we have
$$H^0(X, L^{-1})=0.$$
\end{mylem}
\begin{proof}
We argue by contradiction. Let $s$ be a non-zero section in $H^0(X, L^{-1}).$
Consider the function $\log |s|^2_{L^{-1},h^{-1}}$.
Let $\varphi$ be the local weight of $h$ such that $h= e^{-\varphi}$ locally.
Thus the above function can be locally written as
$\log |s|^2+\varphi$.
In particular, it is a psh function on $X$.
Since $X$ is compact, the only psh functions are the constant functions.
On the other hand, 
$$i \d \dbar \log |s|^2_{L^{-1},h^{-1}}=[s=0]+i \Theta_{L,h}=0$$
where $[s=0]$ is the current associated to the (possible trivial) divisor $s=0$ and $i \Theta_{L,h}$ is the curvature of $(L,h)$.
Since both $[s=0]$ and $i \Theta_{L,h}$ are positive currents, they are 0.
In particular, $s$ never vanishes on $X$, which contradicts the fact that $L$ is a non-trivial line bundle.
\end{proof}
A classical result on nef line bundles is the following. 
Let $A+B$ be a nef line bundle over a compact manifold $X$ where $A, B$ are effective $\R$-divisors without intersection. Then $A, B$ are both nef divisors.
In the case of nefness in higher codimension, we have the following generalized version.
\begin{mylem}
Let $A+B$ be a line bundle that is nef in codimension $k$ over a compact manifold $X$ $($by this, we mean that $c_1(A+B)$ is nef in codimension $k)$, where $A,B$ are effective $\R$-divisors without intersection. Then the divisors $A,B$ are both nef in codimension $k$.

More generally, let $\alpha+c_1(E)$ be a class that is nef in codimension $k$ over a compact manifold $X$, where $E$ is an effective $\R$-divisor and $E_{nn}(\alpha) \cap E = \emptyset$.
Then $\alpha$ is nef in codimension $k$.
\end{mylem}
\begin{proof}
Fix $\alpha_0 \in \alpha, \beta_0 \in c_1(E)$ two smooth representatives.
By assumption, for any $\varepsilon >0$, there exists a quasi-psh function $\varphi_\varepsilon$ on $X$  with analytic singularities such that
$$\alpha_0+i \d \dbar \varphi_\varepsilon \geq -\varepsilon \omega$$
where $\omega$ is some Hermitian metric on $X$ (not necessarily K\"ahler). (For example, we can take a regularisation of the minimal potential $\varphi_{\min, -\frac{\varepsilon}{2}}$.)
We can assume that the singular set of $\varphi_\varepsilon$ has an empty intersection with $V_E$.
Here $V_E$ is some small tubular neighbourhood of $E$.

Let $\psi_\varepsilon$ be a family of quasi-psh functions on $X$  with analytic singularities such that
$$\alpha_0+\beta_0+i \d \dbar \psi_\varepsilon \geq -\varepsilon \omega.$$
We can assume that the singular set of $\psi_\varepsilon$ has codimension at least $k+1$.

Let $\varphi_E$ be a quasi-psh function on $X$ such that $\beta_0+i \d \dbar \varphi_E = [E]$ where $[E]$ is the current associated to $E$.
But definition, the pole of $\varphi_E$ is exactly the support of $E$.
In particular we have that $\psi_\varepsilon-\varphi_E$ is a well-defined quasi-psh function outside $E$ such that
$$\alpha_0+i \d \dbar (\psi_\varepsilon-\varphi_E) \geq -\varepsilon \omega$$
on $X \smallsetminus E$.

Now we glue the potentials to get a quasi-psh function $\Phi_\varepsilon$ with analytic singularities on $X$, such that
$$\alpha_0 +i \d \dbar \Phi_\varepsilon \geq -\varepsilon \omega.$$
We also demand that the singular set of $\Phi_\varepsilon$ be included in the singular set of $\psi_\varepsilon$.
This will finish the proof of the lemma.

On $X \smallsetminus V_E$ we define $\Phi_\varepsilon=\max(\psi_\varepsilon-\varphi_E, \varphi_\varepsilon+C_\varepsilon)$ where $C_\varepsilon$ is a constant which will be determined latter.
In particular, on $X \smallsetminus V_E$ we have
$$\alpha_0+i \d \dbar \Phi_\varepsilon \geq -\varepsilon \omega.$$
On $V_E$, we define $\Phi_\varepsilon= \varphi_\varepsilon+C_\varepsilon$.
On $X \smallsetminus V_E$, $\varphi_E$ is bounded and $\psi_\varepsilon$ is bounded from above.
Near the boundary of $V_E$, $\varphi_\varepsilon$ is also bounded since the singular set of $\varphi_\varepsilon$ has empty intersection with $V_E$.
Thus for $C_\varepsilon$ large enough near the boundary of $V_E$ $\psi_\varepsilon-\varphi_E < \varphi_\varepsilon+C_\varepsilon$.
In particular, $\Phi_\varepsilon$ is a global well defined quasi-psh function such that $\alpha_0+i \d \dbar \Phi_\varepsilon \geq -\varepsilon \omega.$
The singular set of $\Phi_\varepsilon$ in $X \smallsetminus V_E$ is included in the singular set of $\psi_\varepsilon$. On $V_E$, $\Phi_\varepsilon$ is smooth.
This finishes our construction.
\end{proof}
\begin{myrem}
{\em
The condition that the intersection is empty is necessary for the lemma. Otherwise, we have the following counter-example.

The construction uses Cutkosky's construction detailed in the next section. Let $Y$ be a projective manifold such that $\cN_Y=\cE_Y$. Let $\beta \in H^{1,1}(Y, \R)$ be a non psef class. Let $A_1, A_2$ be very ample divisors on $Y$.
Define
$$t_0:= \min \{t|\beta+tc_1(A_1)\; \mathrm{nef}\}.$$
We can assume that $\beta+t_0 c_1(A_2)$ is nef.
Define $X:= \P(A_1 \oplus A_2)$ and denote by $\pi: \P(A_1 \oplus A_2) \to Y$ the natural projection.
By Proposition 4 in the next section, $\pi^* \beta+t_0c_1(\cO(1))$ and $c_1(\cO(1))$ are nef.
Notice that $\cO(1)$ is an effective divisor since $H^0(X,\cO(1) )=H^0(Y, A_1 \oplus A_2) \neq 0$.

By Proposition 5, for any $t<t_0$, $\nu(\pi^* \beta+tc_1(\cO(1)), \P(A_2))>0$ and $E_{nn}(\pi^* \beta+tc_1(\cO(1))=\P(A_2)$.  
This shows that for any $t<t_0$, $\pi^* \beta+tc_1(\cO(1))$ is not nef in codimension 1.
In other words, the nef class $\pi^* \beta+t_0c_1(\cO(1))$ is a sum of not nef in codimension 1 class $\pi^* \beta+tc_1(\cO(1))$ and an effective divisor $(t_0-t)\cO(1)$.
Let $(s_1,s_2) \in H^0(Y,A_1) \oplus H^0(Y, A_2)=H^0(X, \cO(1))$ be a non-trivial section.
Then we have
$$V(s_1,s_2)=\{(x,\xi^*)| \xi^* \in (A_1 \oplus A_2)^*, \; \xi^*(s_1,s_2)=0\}.$$
Identify $\P(A_2)$ as $Y$,then we have $V(s_1,s_2)\cap \P(A_2)=V(s_2) \neq \emptyset$.
Similar calculation shows that for any $E\in |\cO(m)|$ for any $m \in N^*$ 
$$E_{nn}(\pi^* \beta+tc_1(\cO(1)) \cap E \neq \emptyset.$$
}
\end{myrem}
Now we give a version of the Kawamata-Viehweg vanishing theorem over a compact K\"ahler manifold.
\begin{mythm}
Let $(X,\omega)$ be a compact K\"ahler manifold of dimension $n$ and $L$ a nef in codimension 1 line bundle on $X$. Assume that $\langle L^2 \rangle \neq 0$. Assume that there exists an effective $\N$-divisor $D$ such that $c_1(L)=c_1(D)$. Then
$$H^q(X,K_X+L) = 0$$
for $q \ge n-1$.
\end{mythm}
\begin{proof}
In the case $q = n$, we have
$H^n(X,K_X+L)=H^0(X,-L)^*$ by Serre duality. For $L$ psef, $-L$ has
no section unless $L$ is trivial by Lemma 5.
Since $\langle L^2 \rangle \neq 0$, $L$ is not trivial.
Therefore the only interesting case
is $q = n-1$. We divide the proof into two cases.\medskip

\noindent
{\it Case 1}\/: We assume that $L=D$. 
Since the canonical section of $D$ induces a positive singular metric on $(L,h)$ with multiplier ideal sheaf $\cI(h) \subset \cO(-D)$.
In fact we have equality outside an analytic set whose all irreducible components have codimension larger than 2.
Write $D= \sum_i n_i D_i$ where $n_i  \ge 0$ and $D_i$ are the irreducible components of $D$.
Define 
$$Y=(D_{\mathrm{red}})_{\mathrm{Sing}}= \bigcup_{i \neq j} (D_i \cap D_j) \cup \bigcup_i D_{i, \mathrm{Sing}}$$
where $D_{i, \mathrm{Sing}}$ means the singular part of $D_i$.
It is easy to see that we have an equality outside $Y$ and that each irreducible components of $Y$ is of codimension larger than 2.

In particular, the short exact sequence
$$0 \to \cI(h) \to \cO(-D) \to \cO(-D)/\cI(h) \to 0$$
induces that
$$H^{n-1}(X, K_X+L \otimes \cI(h)) \to H^{n-1}(X, K_X+L-D) \to H^{n-1}(X, K_X+L \otimes \cO(-D)/\cI(h))=0$$
since the support of $\cO(-D)/\cI(h)$ is included in $Y$.

Denote by $h_{\min}$ the minimal metric on $L$ where we have a natural inclusion of $\cI(h) \subset \cI(h_{\min})$.
Thus we have the following commuting diagram
\[ \begin{tikzcd}
H^{n-1}(X, K_X+L \otimes \cI(h)) \arrow{r}{} \arrow[swap]{d}{} & H^{n-1}(X, K_X+L \otimes \cI(h_{\min})) \arrow{d}{} \\
H^{n-1}(X, K_X+L-D) \arrow{r}{}& H^{n-1}(X, K_X+L).
\end{tikzcd}
\]
By Corollary 2 proved in Section 3 and the condition that $\nd(L)\ge 2$, we know that the morphism 
$$ H^{n-1}(X, K_X+L \otimes \cI(h_{\min})) \to  H^{n-1}(X, K_X+L)$$
is the $0$ map. Since the left vertical arrow is surjective
in the above diagram, we conclude that the morphism
$$H^{n-1}(X, K_X+L -D) \to  H^{n-1}(X, K_X+L)$$
is also the 0 map. Thus the short exact sequence
$$
0 \to  K_X + L-D \to K_X + L \to K_X + L \vert_D = K_{D} \to 0
$$
gives in cohomology
$$
H^{n-1}(X, K_X+L-D) \to H^{n-1}(X,K_X + L) \to H^{n-1}(D,K_{D}) \simeq 
H^0(D,\cO_{D}) \to H^n(X,K_X+L-D)  \to 0.
$$ 
On the other hand, $H^n(X,K_X+L-D) \simeq H^0(X, \cO_X) \simeq \C$.
Therefore we need only show that
$$
h^0(D, \cO_{D}) = 1.
$$ 
More precisely, $D$ is an effective Cartier divisor in the manifold $X$.
Notice that $D$ is a (possibly non reduced) Gorenstein variety.
In this case, adjunction gives the dualizing sheaf $K_D$ as $(K_X+D)|_D$.
Moreover, Serre duality holds in the same form as in the smooth case.

To calculate the dimension of global sections of $D$, first we show that $D$ is connected.
In fact, 
otherwise we would have $D = A+B$ with $A$ and $B$ effective non-trivial divisors such that $A \cap B = \emptyset$.
In particular we have
$(A \cdot B \cdot \omega^{n-2})= 0$.
But $A$ and $B$ are necessarily nef in codimension 1 by Lemma 6.

We recall the Hodge Index Theorem on a compact K\"ahler manifold $(X, \omega)$ as Theorem 6.33 and 6.34 in \cite{Voi}.
By the Hard Lefschetz theorem, we have
$$H^2(X, \C)= \C\{ \omega\} \oplus H^2(X, \C)_{\mathrm{prim}}$$
where $H^2(X, \C)_{\mathrm{prim}}$ means primitive classes.
The intersection form $(\alpha, \beta) \mapsto (\alpha \cdot \beta \cdot \omega^{n-2})$ has signature $(1, h^{1,1}(X)-1)$ on $H^{1,1}(X)$ since $H^2(X, \C)_{\mathrm{prim}}$ is orthogonal to $\omega$ and the intersection form is negative definite on $H^2(X, \C)_{\mathrm{prim}}$.

On the other hand, by Lemma 4, we have that
$$(A \cdot A \cdot \omega^{n-2}) \geq \langle A \cdot A \cdot \omega^{n-2}\rangle \ge 0$$
and similar inequality for $B$.
We also notice that 
$$(L \cdot L \cdot \omega^{n-2}) \geq \langle L \cdot L \cdot \omega^{n-2}\rangle > 0.$$
Since the intersection form (unlike the movable intersection) is bilinear, we have either $(A \cdot A \cdot \omega^{n-2}) >0$ or $(B \cdot B \cdot \omega^{n-2}) >0$.
Without loss of generality, assume that $(A \cdot A \cdot \omega^{n-2}) >0$.
Thus $B \in A^{\perp}$ and  $(B \cdot B \cdot \omega^{n-2}) \geq 0$.
The Hodge index theorem implies that $B=0$ which is
a contradiction to our assumption.
Hence $D$ is connected, and if 
$h^0(D,\cO_{D}) \geq 2$, then $\cO_{D}$ contains a nilpotent section $t
\neq 0$. 
In other words, the pull back of $t$ under the natural morphism $D_{\mathrm{red}} \to D$ is 0 but lies as a non trivial section in $H^0(D_{\mathrm{red}}, \cO(-\sum_{j \in I}\mu_jD_j))$ for some $1 \leq \mu_j \leq n_j$ for all
$j$.
Let 
$$ J := \{ j \in J \vert {{n_j} \over {\mu_j}} \; {\rm maximal}\}$$
and let $c = {{n_j} \over {\mu_j}}$ be the maximal value. 
Notice that $\mathrm{div}(t)|_{D_i}=- \sum_{j \in I}\mu_j D_j \vert_{ D_i}$ is effective
(possibly $0$) for all $i$. 
We claim that it is impossible that $c =   {{n_j} \over {\mu_j}}$
for all $j \in I$. 
Otherwise, $ L \vert_{ D_i} =c \sum \mu_j D_j \vert_{ D_i} $ 
is psef. ($L$ is nef in codimension 1, so its restriction to any prime divisor is psef.)
Its dual is effective, hence $L \vert_{ D_i} \equiv 0$ for all $i$.
This implies that $ (L \cdot L \cdot \omega^{n-2})= 0$, contradiction.

Thus we find some $j$ such that 
$$
c>{{n_j} \over {\mu_j}}.
$$ 
By connectedness of $D $ we can choose $i_0 \in J$ in such a way that 
there exists $j_1 \in I \smallsetminus J$ with $D_{i_0} \cap D_{j_1} \ne
\emptyset$. 
Now 
$$
\sum_{j \in I}(n_j - c \mu_j)D_j \vert_{ D_{i_0}}
$$
is pseudo-effective as a sum of a psef and an effective line bundle
(this has nothing to do with the choice of $i_0$). Since the sum, taken over
$I$ is the same as the sum taken over $I \smallsetminus \{i_0\}$; we conclude
that 
$$
\sum_{j \ne i_0}(n_j - c\mu_j)D_j \vert_{ D_{i_0}}
$$
is pseudo-effective, too. Now all $n_j- c \mu_j \leq 0$ and $n_{
j_1} - c \mu_{j_1} < 0$ with $D_{j_1} \cap D_{i_0} \ne \emptyset$,
hence the dual of
$$
\sum_{j \ne i_0}(n_j - c\mu_j)D_j \vert D_{i_0}
$$
is effective and non-zero, a contradiction. This finishes the proof
of Case 1.\medskip

\noindent
{\it Case 2}\/: general case. We can write 
$$
L = D + L_0
$$
where $L_0^m \in \Pic^0(X)$ (The exponent $m$ is there because there
might be torsion in $H^2(X,\Z)$; we take $m$ to kill the denominator
of the torsion part). We may, in fact, assume that $m = 1$; otherwise, we pass to a finite \'etale cover $\tilde X$ of $X$ and argue there
(the vanishing on $\tilde X$ clearly implies the vanishing on $X$ by Leray spectral sequence).
In other words, we write $L$ as a sum of $D$ and a flat line bundle $(L_0,h_0)$.
Here $h_0$ is the flat metric.
Thus there exists a bijection between singular positive metrics on $L$ and those on $D$, via the tensor product by $h_0$.  
In particular, the minimal metric on $L$ is the minimal metric on $D$, tensored by $h_0$.

The short exact sequence used above is modified into
$$ 0 \to  K_X + L-D \to K_X + L \to (K_X + L)\vert_D 
= (K_{D} + L_0) \vert_D \to 0.  $$ 
Taking cohomology as before and using a similar discussion, the arguments come down to proving 
$$
H^0(D,-L_0 \vert_D) = 0
$$
since $H^n(X, K_X+L-D) \simeq H^0(X, -L_0) =0$.

The argument on the connectedness of $D$ still works since the arguments only involve the first Chern class and since $L_0$ has no contribution in the first Chern class.
If $-L_0 \vert_D \ne 0$, then we see as above that $-L_0 \vert_D$
cannot have a nilpotent section.
Since $L_0$ is flat, adding a multiple of $L_0$ does not change the pseudo-effectiveness.
By adding a suitable such multiple, the arguments on the non-existence of nilpotent section are still valid.

So if $
H^0(D,-L_0 \vert_D) = 0
$ fails, then $-L_0 \vert_D$
has a section $s$ such that $s \vert_{D_{\mathrm{ red}}}$ has no zeroes.
In other words $-L_0 \vert_{D_{\mathrm{ red}}}$ is trivial. But then $-L_0 \vert_D$ is
trivial,
since the nowhere vanishing section of $H^0(X, -L_0 \otimes_{\cO_X} \cO_X / \cI_{D_{\mathrm{red}}})$ is mapped to a nowhere vanishing section in $H^0(X, -L_0 \otimes_{\cO_X} \cO_X / \cI_{D})$ by passing to the quotient. 

Now let $\alpha: X \to \Alb(X)$ be the Albanese map with image $Y$.
Then $L_0 = \alpha^*(L_0')$ with some line bundle $L_0'$ on $\Alb(X)$.
(We observe that $\Pic^0(X) \cong \Pic^0(\Alb(X))$.)
Notice that $L_0'$ is a non trivial line bundle with $c_1(L_0')=0$. 
Since $L_0 \vert_D$ is
trivial and $L_0$ is non trivial, we conclude that $\alpha (D) \neq Y$.
We claim that $\alpha(D)$ is contained in some proper subtorus $B$ of $\Alb(X)$.

The reason is as follows. 
Let $\nu: \tilde{X} \to X$ be a modification such that $\nu^* (D)$ is a SNC divisor.
Denote by $E_j$ the irreducible components of $\nu^* (D)$.
Define $S \subset \prod_{i} \Pic^0(E_i)$ the connected component containing $(\nu^* L_0|_{E_i})$ of
$$ \{(L_i) \in \prod_i \Pic^0(E_i)| L_i|_{E_i \cap E_j}= L_j|_{E_i \cap E_j}\}.$$
By Proposition 1.5 of \cite{BL}, $S$ is a subtorus since $S$ is a translation of the kernel of
$$\prod_{i}  \Pic^0(E_i) \to \prod_{i,j, i \neq j} \Pic^0(E_i \cap E_j)$$
$$(L_i) \mapsto (L_i|_{E_i \cap E_j}- L_j|_{E_i \cap E_j}).$$
Notice that $\Pic^0(E_i)$ is a torus by Hodge theory since $E_i$ is smooth.
The natural group morphism of $\Pic^0(X) \to S$ given by $L \mapsto (\nu^* L|_{E_i})$ induces by duality the following commuting diagram
$$
\begin{tikzcd}
\prod_{i}  \Pic^0(E_i)^* \arrow[r]\arrow[dr]& S^*\arrow[d]\\
 & (\Pic^0(X))^* \cong \Alb(X).
 \end{tikzcd}$$
Since $L_0 \in S$ is non-trivial, the image of $S^*$ as a complex torus is a proper subtorus in $\Alb(X)$.
We denote its image as $B$.
(Let us observe that by Proposition 1.5 of \cite{BL}, the image of a homomorphism of complex tori is a subtorus.)

Consider the induced
map
$$
\beta: X \to \Alb(X)/B
$$
and denote its image by $Z$. ($Z$ can be singular!)
The image $\beta (D)$ is a point $p$ by construction.
Let $U$ be a Stein neighbourhood of $p$ in $Z$(or some coordinate chart of $p$).
Denote by $m_p$ the maximal ideal of $p$ in $Z$. 
In particular, for any $k \in \N^*$, $m_p^k$ is globally generated on $U$ (by Cartan theorem A).

Let $D=\sum_i n_i D_i$ and define $n_{\max}:=\max (n_i)$. 
Then we have the inclusion $\beta^*H^0(U, m_p^{n_{\max}}) \subset H^0(D,\cO(-n_{\max} D_{\mathrm{red}})|_D) \subset  H^0(D,\cO(-D)|_D)$
where the second inclusion uses the fact that $n_{\max} D_{\mathrm{red}}-D$ is an effective divisor in~$X$.
In particular, for any $i$, $H^0(D_i,\cO(-D)|_{D_i}) \neq 0$.
On the other hand, $\cO(D)|_{D_i}$ is psef since $D$ is nef in codimension 1.
(Observe that nefness is a numerical property. Since $c_1(L_0)=0$, $D$ is nef in codimension 1 as $L$ is.)
By Lemma 5, $D|_{D_i}$ is trivial.
 
Thus we have for any $i$
$$(D \cdot D_i \cdot \omega^{n-2})=\int_{D_i} c_1(D|_{D_i}) \wedge \omega^{n-2}=0.$$
This implies that $(L^2 \cdot \omega^{n-2})=(D^2 \cdot \omega^{n-2})=0$.
On the other hand, since $L$ is nef in codimension 1, 
$(L^2 \cdot \omega^{n-2})\geq  \langle
L^2 \cdot \omega^{n-2} \rangle$.
But this is a contradiction with our assumption.
%We claim that $\dim Z = 1$. The reason is as follows.
%Since $\beta (D)$ is a point, $D$ is included in some
%multiples of a fiber $F=\beta^*(p)$ of $\beta$.
%In other words, there exist a number $m \in \N^*$ and some $D'$ psef divisor such that $mF=D+D'$ as cohomology class.
%By the superadditivity of movable intersection, we have
%$$m^2 \langle F^2 \cdot \omega^{n-2} \rangle \geq \langle D^2 \cdot \omega^{n-2}\rangle.$$
%On the other hand, $F$ is nef as pull back of nef class.
%Thus $(F^2 \cdot \omega^{n-2})=\langle F^2 \cdot \omega^{n-2} \rangle$.
%Since $\dim Z=1$, 
%$$(F^2 \cdot \omega^{n-2})=\int_X \beta^* (c_1(p))^2 \wedge \omega^{n-2}=0.$$
%But this contradicts $\langle
%D^2 \rangle \neq 0$. 
\end{proof}
\begin{myrem}
{\em
If $D$ is a smooth reduced divisor, we can also argue as follows at the end of Case 2. We observe that
$L_0$ is a non-trivial element in a translation of the kernel of $\Pic^0(X) \to \Pic^0(D)$. 
%In particular, the kernel is not trivial.
On the other hand, we have
$$H^{n-1}(X, K_X+D)=H^1(X, -D)=0 \to H^1(X, \cO_X) \to H^1(D, \cO_D)$$
since by Case 1, $H^{n-1}(X, K_X+D)=0$.
However, $H^1(X, \cO_X) \to H^1(D, \cO_D)$ is the tangent map of $\Pic^0(X) \to \Pic^0(D)$.
By Proposition 1.5 of \cite{BL}, the kernel is discrete.
Moreover, the connected component containing the zero point of the kernel is of finite index in the kernel. In particular, $L_0$ is a torsion element.
This yields a contradiction.
}
\end{myrem}
\section{Examples and counter-examples}
In this section, we first give for each $k \in \N^*$ an example of a psef class $\alpha_k$ on some manifold $X_k$, such that $\alpha_k$ is nef in codimension $k$ but not nef in codimension $k+1$.
This shows in particular that the inclusion of the various types of nef cones can be strict.

For the convenience of the reader, we recall Cutkosky's construction described in \cite{Bou04}, as well as all needed material for our use.

Let $\cE$ be a vector bundle of rank $r$ over a manifold $Y$ and $L$ be a line bundle over $Y$.
Since there exists a surjective bundle morphism given by projection $\cE \oplus L \to \cE$, 
we can view $D:=\P(\cE)$ as a closed submanifold of $\P(\cE \oplus L)$.
Note that the restriction of $\cO_{\P(\cE \oplus L)}(1)$ on $\P(\cE)$ is the tautological line bundle $\cO_{\P(\cE)}(1)$.
We notice that the canonical line bundle of the projectivization of a vector bundle $\P(\cE)$ is given by
$$K_{\P(\cE)} =\cO_{\P(\cE)}(-(r+1)) + \pi^*(K_Y +\mathrm{det} \cE)$$
where $\pi: \P(\cE) \to Y$ is the projection.
From the short exact sequence
$$0 \to T_{\P(\cE)} \to T_{\P(\cE \oplus L)}|_{\P(\cE)} \to N_{\P(\cE)/ \P(\cE \oplus L)}=\cO(D)|_{\P(\cE)} \to 0$$
we have
$$K_{\P(\cE \oplus L)}|_{\P(\cE)} =K_{\P(\cE)} \otimes \cO(-D)|_{\P(\cE)}.$$
Using the formula for the canonical line bundle, we have
$$\cO(1)|_{\P(\cE)}=(\cO(D) \otimes \pi^* L)|_{\P(\cE)}.$$
We observe that by the Leray-Hirsh theorem for Bott-Chern cohomology, 
$$H^{1,1}_{BC}(\P(\cE \oplus L), \R)=\R c_1(\cO(1)) \oplus \pi^* H^{1,1}_{BC}(Y, \R).$$
In particular, this implies that the inclusion $i: \P(\cE) \to \P(\cE \oplus L)$ induces an isomorphism on $H^{1,1}_{BC}$.
Hence we find that on $\P(\cE \oplus L)$
$$c_1(\cO(1))=c_1(\cO(D))+\pi^* c_1(L).$$
Now let $Y$ be a compact complex manifold of dimension $m$ and $L_0, \cdots, L_r$ the line bundles over $Y$.
We define 
$$X := \P(L_0 \oplus \cdots \oplus L_r).$$
We denote $H:= \cO(1)$ the tautological line bundle over the projectivization and $h:=c_1(H)$.
For any $i$, the projection $L_0 \oplus \cdots L_r \to L_0 \oplus \cdots \hat{L_i} \oplus \cdots \oplus L_r$ induces inclusions of hypersurfaces
$$D_i:= \P(L_0 \oplus \cdots\oplus \hat{L_i} \oplus \cdots \oplus L_r).$$
By the above discussion
$$d_i +l_i=h$$
where $d_i:=c_i(\cO(D_i))$ and $l_i:=c_1(L_i)$.
In fact, by calculating the transition function, we can show that $\cO(1)$ is linear equivalent to $L_i + D_i$.
However, the relation of Chern classes is enough for our use here.

We have the following explicit description of nef cone and psef cone in this case.
We denote by $\rC$ the cone generated by the $l_i$.
\begin{myprop}
Let $\alpha \in H^{1,1}_{BC}(X, \R)$ be a class that is decomposed as $\alpha= \pi^* \beta+\lambda h$. Then

(1) $\alpha$ is nef iff $\lambda \geq 0$ and $\beta + \lambda \rC$ is contained in $\cN_Y$.

(2) $\alpha$ is psef iff $\lambda \geq 0$ and $(\beta +\lambda \rC) \cap \cE_Y \neq \emptyset$.
\end{myprop}
\begin{proof}
We notice that if $\alpha$ contains a positive current $T=\theta + i \d \dbar \varphi$ with $\theta$ smooth, then the pluripolar set $P(\varphi)=\{\varphi=-\infty\}$ is of Lebesgue measure 0. Hence, by the Fubini theorem, the set
$$\{y \in Y, \pi^{-1} (y) \subset P(\varphi)\}$$
is of Lebesgue measure 0.
For $y$ outside the measure 0 set, $\alpha|_{\pi^{-1}(z)}$ is the class of $T|_{\pi^{-1}(z)}$.
It is also equal to the class of $\lambda c_1(\cO_{\P^r}(1))$, and this implies that $\lambda \geq 0$. We always assume in the following that $\lambda \geq 0$.

(1) If $\alpha$ is nef, the restriction of $\alpha$ to $\P(L_i)$ for any $i$ is also nef where $\P(L_i)$ is biholomorphic to $Y$ given by $\pi$.
Note that $\alpha|_{\P(L_i)}=\lambda l_i + \beta$ is nef as a restriction of nef class.
So $\beta + \lambda \rC$ is contained in $\cN_Y$.

On the other hand, $\alpha= \pi^* \beta +h=\pi^*(\beta + \lambda l_i)+\lambda d_i$ for any $i$ where $\beta + \lambda l_i$ is nef by assumption. Hence the non-nef locus of $\alpha$ is contained in $D_i$. 
Since the intersection of all $D_i$ is empty, we conclude that $\alpha$ is nef.

(2) Let $t_i \in [0,1]$ such that $\sum t_i =1$ 
and $\beta+\sum_{i=0}^r t_i l_i \in \cE_Y$.
Hence $h= \sum t_i h=\sum t_i \pi^* l_i+ \sum t_i d_i$ and 
$\alpha= \pi^* (\beta + \lambda \sum t_i l_i)+ \lambda \sum t_i d_i$.
$d_i$ is psef since it contains the positive current associated to $D_i$.
As a sum of psef classes, $\alpha$ is psef.

For the other direction, we argue by induction. When $r=0$, $X=Y$ and $\alpha= \beta +\lambda l_0$.
By the assumption that $\alpha$ is psef, we have
$$\alpha \in (\beta +\lambda \rC) \cap \cE_Y.$$
Continue the induction on $r$.
Let $T$ be a closed positive current in $\alpha$.
we have that $\alpha -\nu(T, D_0)d_0$ is psef containing the current $T - \nu(T,D_0)[D_0]$.
And $(\alpha -\nu(T, D_0)d_0)|_{D_0}$ is psef since the restriction of the current $T - \nu(T,D_0)[D_0]$ on $D_0$ is well defined.
Now $D_0$ is the projectivisation of a vector bundle of rank $r$ over $Y$.
As a cohomology class 
$$\alpha-\nu(T, D_0)d_0= \pi^* (\beta +\lambda l_0)+(\lambda-\nu(T,D_0))d_0$$
Restrict $\alpha$ on some fibre of $\pi$ as above.
We have that $\lambda \geq \nu(T, D_0)$.
By induction, we see that the psef class $(\alpha -\nu(T, D_0)d_0)|_{D_0}$, which is also equal to
$\pi^*(\beta +\nu(T, D_0)l_0)+(\lambda-\nu(T, D_0))h$,
satisfies
$$(\beta +\nu(T, D_0)l_0 +(\lambda-\nu(T, D_0))\rC_0 )\cap \cE_Y \neq \emptyset$$
where $\rC_0$ is the cone generated by $l_1, \cdots, l_r$.
In other words,
$$(\beta +\lambda \rC) \cap \cE_Y \neq \emptyset.$$
\end{proof}
We will also need the following explicit calculation of the generic minimal multiplicity in this example.
From now on, we choose $Y$ such that the nef cone $\cN_Y$ and the psef cone $\cE_Y$ coincide (for example, we can take $Y$ to be a Riemann surface).

We denote $I$ a subset of $\{0, \cdots ,r\}$ with complement $J$.
We denote $V_I := \bigcap_{i \in I} D_i=\P(\oplus_{j \in J}L_j)$ and $\rC_I$ the convex envelope of $l_i(i \in I)$.

We observe that the non-nef locus of any psef class is contained in the union of $D_i$.
The reason is as follows:
since $\alpha= \pi^* \beta + \lambda h$ is psef, by Proposition 4 we know that there exist $t_i \in [0,1]$ with $\sum t_i=1$ such that $\beta +\lambda(\sum t_i l_i) \in \cE_Y=\cN_Y$. Hence
$$\alpha=\pi^*(\beta +\lambda(\sum t_i l_i))+\lambda(\sum t_i d_i)$$
is a sum of nef divisor and effective divisor.
(Since $\alpha$ is psef, $\lambda \geq 0$.)
So the non-nef locus of $\alpha$ is contained in the union of $D_i$.
\begin{myprop}
Let $\alpha$ be a big class such that $\alpha= \pi^* \beta + \lambda h$.
The generic minimal multiplicity of $\alpha$ along $V_I$ is equal to
$$\nu(\alpha, V_I)= \min\{t \geq 0, (\beta +t C_I+(\lambda-t)C_J)\cap N_Y \neq \emptyset\}.$$ 
More precisely, we have $\nu(\alpha, V_I)=\nu(\alpha,x)$ for any $x \in V_I \smallsetminus \bigcup_{j \in J}D_j$.
\end{myprop}
\begin{proof}
Let $\mu: X_I \to X$ the blow-up of $X$ along $V_I$ with exceptional divisor $E_I$.
Hence we have $E_I=\P(N^*_{V_I/X})$ with
$N^*_{V_I/X}=\oplus_{i\in I} \cO_{V_i}(-D_i)$.
By Lemma 1, we get
$$\nu(\alpha, V_I)=\nu(\mu^* \alpha, E_I).$$
Denote by $H_I$ the tautological line bundle over $E_I$ where we have $\cO_{E_I}(-E_I)=H_I$.

For $t \geq 0$, the restriction of $\mu^* \alpha -t c_1(\cO(E_I))$ to $E_I$ is psef and is hence equivalent to that $\mu^* \alpha +tc_1(H_I)$ is psef.
By Proposition 4, the latter is equivalent to the fact that $\alpha+ t \rC(\pi^* l_i-h)=\alpha-th +t \pi^* \rC(l_i)$
intersects $\cE_{V_I}$ where $\rC(l_i)$ is the convex envelop of $l_i$ ($i\in I$).
Note also that 
$$\alpha-th +t \pi^* \rC(l_i)=\pi^*(\beta +t \rC(l_i))+(\lambda-t)h $$
where we denote by the same notation $\pi$ to be the projection from $V_I$ to $Y$ and $h$ to be the first Chern class of the tautological line bundle over $V_I$.
By Proposition 4, it is psef if and only if 
$\beta +t\rC_I +(\lambda-t)C_J$ intersects the psef cone $\cE_Y$.

Since the class $\mu^* \alpha- \nu(\alpha,V_I)c_1(\cO(E_I))$ has positive current
$\mu^* T_{\min} -\nu(T_{\min},V_I)[E_I]$
whose restriction to $E_I$ is well defined by Siu's decomposition theorem.  
By the last paragraph we have
$$\nu(\alpha, V_I)=\nu(T_{\min }, V_I) \geq \min\{t \geq 0, (\beta +t C_I+(\lambda-t)C_J)\cap N_Y \neq \emptyset\}.$$
On the other direction, let $\gamma:= \beta +t \sum_{i \in I}a_i l_i+(\lambda-t)\sum_{j \in J}b_j l_j$ be a psef (equivalently nef by assumption) class on $Y$ with $\sum a_i=\sum b_j=1$. 
Hence $\alpha=\pi^* \gamma +t \sum a_i d_i +(\lambda-t)\sum b_j d_j$.
For $x \in V_I \smallsetminus \bigcup_{j \in J}D_j$,
$$\nu(\alpha,x) \leq t \sum a_i \nu([D_i], x)+(\lambda-t) \sum b_j \nu([D_j],x)\leq t \sum a_i=t.$$
In particular, this shows that
$$\nu(\alpha, V_I)\leq \min\{t \geq 0, (\beta +t C_I+(\lambda-t)C_J)\cap N_Y \neq \emptyset\}.$$ 
By the proof, the equality is attained for $x \in V_I \smallsetminus \bigcup_{j \in J}D_j$.
\end{proof}
We notice that if we use the algebraic analogue in the projective case as in \cite{Nak}, we can weaken the assumption to the case that $\alpha$ is just a psef class.

In particular, Proposition 5 shows that $\cup D_i$ is stratified by the sets $V_I \smallsetminus \bigcup_{j \in J}D_j$ with respect to the generic minimal multiplicity.

Now we are prepared to give our construction.
Let $Y$ as above be a projective manifold such that the nef cone coincides with the psef cone.
Define $X_k= \P(\cO_Y \oplus \cO_Y(A_1) \oplus \cdots \oplus \cO_Y(A_{k+1}))$ where $A_i$ are the ample line bundles over $Y$.
Let $\beta \in H^{1,1}_{BC}(Y,\R)$ be a not-nef class. 
Denote $H$ be the tautological line bundle over $X_k$ and denote $h$ its first Chern class.
Define $\alpha= \pi^* \beta +h$.
We assume that:

For any $i$, $\beta + c_1(A_i)$ is nef and big.

As above, $\P(\cO_Y) \simeq Y$ is a closed submanifold of $X_k$ of codimension $k+1$ via the projection of $\cO_Y \oplus \cO_Y(A_1) \oplus \cdots \oplus \cO_Y(A_{k+1})\to \cO_Y$.
$\alpha$ is psef but not nef on $X_k$ by Proposition 4.
In fact, if $\alpha$ is nef, its restriction to the submanifold $Y$ (i.e. $\beta$) will be nef.
For any subset $I \neq \{1,\cdots r\}$ (taking $L_0:=\cO_Y$),
by Proposition 5, $\nu(\alpha,V_I)=0$ since $\beta + \sum_{j\in J} c_1(A_j)$ is nef which means we can take $t=0$ on the right hand of the equation.
By Proposition 5, $\nu(\alpha, x)$ is constant on $\P(\cO_Y)$.
The non-nef locus can not be empty; otherwise, $\alpha$ would be nef.
Nevertheless, the non-nef locus has to be contained in $\P(\cO_Y)$.  
Hence the constant cannot be zero.

In conclusion, we have $\nu(\alpha,\P(\cO_Y))>0$, which in particular shows that $\alpha$ is not nef in codimension $k+1$.
On the other hand, the non-nef locus is also $\P(\cO(Y))$ which in particular shows that $\alpha$ is nef in codimension $k$.

With the explicit calculation of generic minimal multiplicity, we discuss the optimality of the divisorial Zariski decomposition.
Take $k=1$ in the above construction.
Take $\beta$ to be the first Chern class of some line bundle.
Hence by the above calculation, $\alpha$ is nef in codimension one but not nef in codimension two.
Its non-nef locus is $\P(\cO_Y)$.
For $\alpha$, there doesn't exist a unique decomposition of this psef class $\alpha=c_1(L)$ into a nef in codimension 2 $\R$-divisor $P$ and an effective $\R$-divisor $N$ such that the canonical inclusion $H^0(\lfloor kP \rfloor) \to H^0(kL)$ is an isomorphism for each $k>0$.
Here the round-down of an $\R$-divisor is defined coefficient-wise.
On the contrary, this decomposition will also be the divisorial Zariski decomposition.
However, $\alpha$ is nef in codimension 1, the uniqueness of the divisorial Zariski decomposition shows that the nef in codimension 2 part has to be $\alpha$ itself.
This is a contradiction.
In particular, when $Y$ is a Riemann surface, it gives an example in dimension three where the classical Zariski decomposition does not exist (although it is always possible in dimension 2).

Consider a psef class $\alpha$ on some compact manifold $X$. In general, there does not always exist a composition of finite blow-up(s) of smooth centres $\mu:  \tilde{X}\to X$ such that the nef in codimension 1 part of $\mu^* \alpha$ is in fact nef.
This example is first shown in \cite{Nak}. 
%Let us briefly revisit his example.

%Assume that it is nef in codimension one but not nef in codimension two.
%We prove by contradiction.
%If there exists a finite composition of blow-up(s) of smooth centres $\mu: \tilde{X} \to X$ such that the nef in codimension 1 part $Z(\mu^* \alpha)$ of $\mu^* \alpha$ is in fact nef.
%We denote by $N(\mu^* \alpha)$ the negative part of $\mu^* \alpha$ in the divisorial Zariski decomposition.
%$\mu_* N(\mu^* \alpha)$ is a finite sum of cohomology classes of the effective divisor(s).
%Since for any prime divisor $D$, 
%$\nu(\mu_* N(\mu^* \alpha),D) \leq \nu(\alpha, D)=0$, ($\alpha$ is nef in codimension 1), in fact 
%$\mu_* N(\mu^* \alpha)=0$.
%On the other hand, by Corollary 4.5, 
%Thus we have
%$\alpha=\mu_* \mu^* \alpha=\mu_* Z(\mu^* \alpha)$ %should be a nef class in codimension 2.
%This contradicts our assumption.
%In particular, this gives a weak version of the fact that there exists no finite composition of blow-up(s) with smooth centres of the variety constructed by Nakayama \cite{Nak} such that, given a psef cohomology class, the pullback of its
%nef in codimension 1 part is nef.

%In fact 
Let $\alpha$ be a big class on a compact K\"ahler manifold $X$.
Assume that there exists no finite composition of blow-up(s) with smooth centres such that the
the nef in codimension 1 part of $\mu^* \alpha$ is in fact nef. 
For example, we can take the pullback of the class constructed by Nakayama on $X$.
%by $p: X \times T \to X$ where $T$ is a complex torus.
We have the following arguments to conclude that in fact there exists no modification such that the
the nef in codimension 1 part of $\mu^* \alpha$ is in fact nef. 
In general, a modification is not necessarily a  composition of blow-up(s) with smooth centres. 
However, by Hironaka's results, any modification is dominated by a finite composition of blow-up(s) with smooth centres. 
In other words, for $\nu: \tilde{X} \to X$ a modification, there exists a commutative diagram
$$
\begin{tikzcd}Y \arrow[rd, "g"]\arrow[r, "f"] & \tilde{X} \arrow[d, "\nu"] \\& X\end{tikzcd}
$$
where $g$ is a finite composition of blow-up(s) with smooth centres and $f$ is holomorphic.
To prove that there exists no modification such that the nef in codimension 1 part of the pullback of some cohomology class is nef,
%by the above argument,
we have to prove that if $Z(\nu^* \alpha)$ is nef, $Z(g^* \alpha)$ is also nef.
This is done by the following proposition.
%This fact is first observed by Nakayama \cite{Nak}.
It shows in particular that in the above example, if $Z(\nu^* \alpha)$ is nef, $ Z(g^* \alpha)=f^* Z(\nu^* \alpha)$ is also nef.

Notice that the initial argument of Nakayama already proves the non-existence of Zariski decomposition for any modification.
%Thus, our arguments still produce a contradiction for the fact that there exists no modification 
%finite composition of blow-up(s) with smooth centres in the variety, 
%such that the part of the pullback of $\alpha$ that is nef in codimension 1 is actually nef.
\begin{myprop}
(1) Let $f: Y \to X$ be a holomorphic map between two compact complex manifolds and $\alpha$ be a psef class on $X$. 
Assume that $Z(\alpha)$ is nef.
Then $f^* N(\alpha)\geq  N(f^* \alpha) $ 
where the inequality relation $\geq$ means the difference is a psef class.

(2) Let $f: Y \to X$ be a modification between two compact complex manifolds and $\alpha$ a big class on $X$. 
Then $N(f^* \alpha)\geq f^* N(\alpha).    $ 
\end{myprop}
\begin{proof}
(1) By the convexity of minimal multiplicity along the subvarieties, 
$$ N(f^* \alpha) \leq  N(f^* N(\alpha)) + N(f^* Z(\alpha)). $$
Since $Z(\alpha)$ is nef, $f^* Z(\alpha)$ is also nef, and thus $ N(f^* Z(\alpha))=0 $.
The conclusion follows observing that $ N(f^*N( \alpha)) \leq f^*N(\alpha) $.

(2) We claim that for any positive current $T \in f^* \alpha$, there exists a positive current $S \in \alpha$ such that $T=f^* S$.
It is proven in Proposition 1.2.7 \cite{Bou04} in a more general setting.
For the convenience of the reader, we give proof in this special case.

Fix a smooth representative $\alpha_\infty$ in $\alpha$.
There exists a quasi-psh function $\varphi$ such that $T= f^* \alpha_\infty+i \d \dbar \varphi$.
Let $U$ be an open set of $X$ such that $\alpha_\infty= i \d \dbar v$ on $U$. 
The function $v \circ f+\varphi$ is psh on $f^{-1}(U)$.
All the fibres are compact and connected (the limit of general connected fibre, the points, is still connected); thus, $v \circ f+\varphi$ is constant along the fibres.
Thus there exists a function $\psi$ on $U$ such that $\varphi=\psi \circ f$.
Since $\varphi$ is $L^1_{\loc}$ and $f$ is biholomorphic on a dense Zariski open set, $\psi$ is also $L^1_{\loc}$.
It is easy to see that $\psi$ is independent of the choice of $v$ and is defined on $X$.
Define $S=\alpha_\infty+i \d \dbar \psi$ and we have $T =f^* S$.

In particular, the minimal current in $f^* \alpha$ is the pull back of the minimal current in $\alpha$ $T_{\min}$.
Thus 
$$N(f^* \alpha)=\{\sum \nu(f^* T_{\min},E)[E]\}\geq \{\sum_{\mathrm{codim}(f(E))=1} \nu(f^* T_{\min},E)[E]\}$$
$$=\{\sum_{\mathrm{codim}(f(E))=1} \nu( T_{\min},f(E))[E]\}=f^* N(\alpha)$$
where the sum is taken over all irreducible hypersurfaces of $Y$.
\end{proof}
%It seems to be hard to use the above easy argument to conclude the general modification case.

Let us point out that a current with minimal singularities does not necessarily have analytic singularities for such a big class $\alpha$ that is nef in codimension one but not nef in codimension two; this has been observed by Matsumura \cite{Mat}. 
The reason is as follows.
In such a situation, there exists a modification $\nu: \tilde{X} \to X$ such that the pullback of $\alpha$ has a minimal current of the form $\beta+[E]$ where $\beta$ is a semi-positive smooth form and $[E]$ is the current associated to an effective divisor supported in the exceptional divisor.
In particular, the sum $\{\beta\}+\{[E]\}$ as cohomology class gives the divisorial Zariski decomposition of the class $\nu^* \alpha$.
Remind that for a big class, the Zariski projection of $\alpha$ is given by
$$\alpha-\sum_{D} \nu(T_{\min},D)\{[D]\}$$
where $D$ runs over all the irreducible divisors on $X$ and $T_{\min}$ is the current with minimal singularity in the class $\alpha$ (cf. Proposition 3.6 of \cite{Bou04}).
On the other hand, the push forward $\nu_*$ and pull-pack $\nu^*$ induces isomorphism between $\nu^* \alpha_\infty$-psh functions on $\tilde{X}$ and $\alpha_\infty$-psh functions on $X$ where $\alpha_\infty$ is a smooth element in $\alpha$.
In particular, the pullback of the minimal current of $\alpha$ is the minimal current in $\nu^* \alpha$, which is also a big class.
Hence $\nu^* \alpha$ admits a divisorial Zariski decomposition where the Zariski projection is semi-positive (hence nef).
This contradicts the last paragraph.
\begin{myrem}
{\em 
As a direct consequence of Matsumura's observation,
the assumption in our Kawamata-Viehweg vanishing theorem that the line bundle is numerical equivalent to an effective integral divisor cannot be deduced from the other conditions.
In the nef case as in \cite{DP}, they reduce from that assumption that the line bundle $L$ is nef with $(L^2) \neq 0$ that
$L$ is numerically equivalent to an effective integral divisor $D$ and that there exists a positive singular metric $h$ on $L$  such that $\cI(h)=\cO(-D)$.

Now we show that for a big line bundle $L$ which is nef in codimension 1 but not nef in codimension 2 over a compact K\"ahler manifold $(X, \omega)$, $\langle L^2 \rangle \neq 0$ and $\frac{i}{2 \pi} \Theta(L,h_{})$ is not a current associated to an effective integral divisor. In particular, Nakayama's example  provides such an example.

By the observation of Matsumura, the curvature current of the minimal metric can not even be a current associated with a real divisor.
Since $L$ is big, $\langle L^n \rangle = \mathrm{Vol}(L) \neq 0$. 
By the Teissier-Hovanskii inequalities, we get
$$\langle L^2 \cdot \omega^{n-2} \rangle =\langle L^2 \rangle \cdot \omega^{n-2} \geq \mathrm{Vol}(L)^{2/n} \mathrm{Vol}(\omega)^{(n-2)/n} >0.$$
This shows in particular that $\langle L^2 \rangle \neq 0$.
}
\end{myrem}
\begin{myrem}
{\em 
Let us observe that this kind of construction can also be used to give an example of the manifold with psef anticanonical line bundle, for which the Albanese morphism is not surjective.

According to the knowledge of the author, this kind of question has been first proposed in \cite{DPS93} where the authors ask whether the Albanese map of a compact K\"ahler manifold is surjective under the assumption that the anticanonical line bundle is nef. The statement has been
proven first by P\u{a}un \cite{Paun12} using the positivity of direct image, and then by Junyan Cao \cite{Cao13} via a different and simpler method. In case the manifold is projective, and the anticanonical divisor is nef, this had been proven earlier by Qi Zhang \cite{Zha}.

Let us use the same notation as above. Take $Y$ to be a complex curve of genus larger than 2. 
By a classical result, the Albanese map of $Y$ is the embedding of the curve into its Jacobian variety $\mathrm{Jac}(Y)$. In particular, the Albanese map is not surjective.
Fix $A$ an ample divisor on $Y$.
Define $E= A^{\otimes p} \oplus A^{ \otimes -q }$ where $p, q \in \N$ will be determined latter.
Denote $X = \P(E)$ with $\pi: X \to Y$. 

We claim that the Albanese morphism of $X$ is the composition of the natural projection $\pi$ and the Albanese morphism of $Y$.
The reason is as follows: (cf. Proposition 3.12 in \cite{DPS94})

Since the fibres of $\pi$ are $\P^1$ which is connected and since $\pi$ is differentially a locally trivial fibre bundle, we have $R^0 \pi_* \R_X= \R_Y$, while $R^1 \pi_* \R_X=0$.
We remark that $H^1(\P^1,\R)=0$.
The Leray spectral sequence of the constant sheaf $\R_X$ over $X$ satisfies
$$E^{s,t}_2=H^s(Y,R^t \pi_* \R_X),E^{s,t}_r \Rightarrow H^{s+t}(X,\R).$$
Since $R^1 \pi_* \R_X=0$, the Leray spectral sequence always degenerates in $E_2$.
(In fact, by \cite{Bl56}, the Leray spectral sequence always degenerates in $E_2$ for K\"ahler fibrations.)
Hence we have
$$H^1(X, \R_X) \cong H^1(Y,\R_Y).$$
Since $Y$ and $X$ is compact K\"ahler, we have by Hodge decomposition theorem that
$$H^0(X, \Omega^1_X) \cong H^0(Y,\Omega^1_Y).$$
Since $\pi^*: H^0(Y, \Omega^1_Y) \cong H^0(X,\Omega^1_X)$ is an injective morphism, it induces an isomorphism.
Passing to the quotient, it induces an isomorphism $\pi^*: \Alb(X) \cong \Alb(Y)$.
The claim is proven by the universality of the Albanese morphism:
\[ \begin{tikzcd}
X \arrow{r}{\pi} \arrow{d}{\alpha_X} & Y \arrow{d}{\alpha_Y} \\%
\Alb(X) \arrow{r}{\pi^*}& \Alb(Y).
\end{tikzcd}
\]
We also claim that for well chosen $p,q$, the anticanonical line bundle $-K_X$ is big but not nef in codimension 1. 
In particular, this shows that there exists a compact K\"ahler manifold $X$ such that $-K_X$ is psef but the Albanese morphism is not surjective. Recall that 
$$K_X=\pi^* (K_Y \otimes \mathrm{det} \; E) \otimes \cO_{X}(-2).$$
%Thus we have for $m \in \N^*$
%$$\pi_* (-mK_X)=S^m E \otimes A^{m(q-p)} \otimes K_Y^{-m}.$$
In particular for $q \gg p$, $-(K_Y \otimes \mathrm{det} \; E)=(q-p)A-K_Y$ is ample. 
On the other hand, $\cO_X(1)$ is big since one of the component in the direct sum bundle $E$ is big.
Thus $-K_X$ is big for $q \gg p$.
On the other hand, the surjective morphism $E \to A^{\otimes p}$ induces the closed immersion $\P(A^{\otimes p})\cong Y \to X$.
We have that $-K_X|_{\P(A^{\otimes p})}= -K_Y-pA$.
For $p$ big enough, we can assume that $-K_Y-pA$ is not psef. 
As a consequence, $-K_X$ is not nef in codimension 1.

In fact, we can calculate the generic minimal multiplicity as
$$\nu(c_1(-K_X),\P(A^{\otimes p}))=\min\{t, -K_Y+(q-p)A+tpA-(2-t)qA \mathrm{\; is \; nef} \}.$$
Since $K_Y$ is ample, we know that the  generic minimal multiplicity along $\P(A^{\otimes p})$ is strictly larger than 1.
In particular, consider any singular metric $h_\varepsilon$ on $-K_X$ such that its curvature satisfies
$i \Theta(-K_X, h_\varepsilon) \geq -\varepsilon \omega$
where $\omega$ is some K\"ahler form on $X$.
Then the multiplier ideal sheaf is not trivial.
Near a point of $\P(A^{\otimes p})$, 
choose some local coordinate such that
$\P(A^{\otimes p})=\{z_1=0\}$. 
By Siu's decomposition,
the local weight of $h_\varepsilon$ is more singular than $\log(|z_1|^2)$.
This implies that
$\cI(h_\varepsilon) \subset \cI_{\P(A^{\otimes p})}$
where $\cI_{\P(A^{\otimes p})}$ is the ideal sheaf associated to $\P(A^{\otimes p})$.

Therefore, some additional condition is certainly needed to ensure the surjectivity of Albanese morphism.
}
\end{myrem}
 
\end{document}